\documentclass{amsart}
\usepackage{amssymb, amstext, amscd, amsmath}
\usepackage[all,cmtip]{xy}
\usepackage{enumitem}
\usepackage{xcolor}
\usepackage{tikz}
\usepackage{tikz-cd}
\usepackage{comment}

\usepackage[nobysame,alphabetic,lite]{amsrefs}

\usepackage[
pdftitle={RDP Fell bundles},
pdfauthor={Alcides Buss, Pradyut Karmakar},
pdfsubject={Mathematics},
colorlinks=true,
linkcolor=blue,
citecolor=blue,
urlcolor=blue
]{hyperref}

\renewcommand*{\MR}[1]{\href{https://mathscinet.ams.org/mathscinet-getitem?mr=#1}{MR #1}}
\newcommand*{\arxiv}[1]{\href{https://arxiv.org/abs/#1}{arXiv:#1}}

\usepackage{verbatim}

\theoremstyle{plain}
\newtheorem{theorem}{Theorem}[section]
\newtheorem{corollary}[theorem]{Corollary}
\newtheorem{proposition}[theorem]{Proposition}
\newtheorem{lemma}[theorem]{Lemma}
\theoremstyle{definition}
\newtheorem{remark}[theorem]{Remark}
\newtheorem{example}[theorem]{Example}

\newtheorem{definition}[theorem]{Definition}

\theoremstyle{plain}

\newcommand{\bC}{{\mathbb{C}}}

\newcommand{\bF}{{\mathbb{F}}}

\newcommand{\bR}{{\mathbb{R}}}

\renewcommand{\phi}{\varphi}
\newcommand{\upchi}{{\raise.35ex\hbox{\ensuremath{\chi}}}}

\newcommand{\Aut}{\operatorname{Aut}}

\newcommand{\supp}{\operatorname{supp}}

\title{Rapid decay and localizability for Fell bundles over \'etale Groupoids}
\author[Alcides Buss]{Alcides Buss}
\address[Alcides Buss]{Departamento de Matem\'atica, Universidade Federal de Santa Catarina, 88.040-900
Florian\'opolis-SC, Brazil}
	\email{alcides.buss@ufsc.br}
	\urladdr{http://mtm.ufsc.br/~alcides/}

\author{Pradyut Karmakar}
\address[Pradyut Karmakar]{Sam Houston State University, 332 G LDB, 1900 Avenue I, Huntsville, Texas, USA}
\email{karmakar.pradyut@gmail.com}

\begin{document}

\subjclass[2020]{Primary 46L55, 22A22; Secondary 46L05, 37B05}
\keywords{Rapid decay property, Fell bundles, \'etale groupoids, reduced $C^*$-algebras, partial actions, localizability, Deaconu-Renault groupoids}

\begin{abstract}
    We introduce a notion of the Rapid Decay Property (RDP) for Fell bundles over locally compact Hausdorff \'etale groupoids, extending earlier rapid decay theories for \'etale groupoids and twists. Our approach yields analytic control on convolution norms and leads to the existence of dense Schwartz-type $*$-subalgebras of the reduced cross-sectional $C^*$-algebra $C_r^*(E)$. As an application, we obtain approximation results showing that, under suitable hypotheses, sections of $C_r^*(E)$ with support contained in an open subset $U\subseteq G$ can be approximated in the reduced norm by compactly supported sections supported inside $U$. In this sense, the Rapid Decay Property provides an analytic mechanism leading to a form of \emph{localizability} for Fell bundles.
    
    We also investigate the relationship between RDP, polynomial growth, and dynamical systems. We show that Fell bundles over groupoids with polynomial growth naturally satisfy the RDP. Furthermore, for a transformation groupoid $G=\Gamma\ltimes_\theta X$ associated with a partial action, we prove that RDP for a Fell bundle over $G$ is equivalent to RDP for a naturally associated Fell bundle over the discrete group $\Gamma$. Finally, we apply these tools to Deaconu-Renault groupoids. By realizing them as partial crossed products of free groups, we show that the presence of persistent branching forces exponential growth, completely obstructing the RDP. This provides a striking illustration of a system where the acting group has RDP, but the associated groupoid fails to inherit it, fully clarifying the boundary between the group and groupoid theories.
\end{abstract}

\maketitle

\tableofcontents

\section{Introduction}

Approximation and decay properties play a fundamental role in the analysis of $C^*$-algebras associated with groups and groupoids. Among these, the Rapid Decay Property (RDP), introduced by Haagerup and Jolissaint for discrete groups, provides analytic control over convolution norms and allows one to construct dense Fr\'echet $*$-subalgebras of reduced group $C^*$-algebras. This property has proved extremely useful in operator algebras, geometric group theory, and noncommutative geometry.

Rapid decay phenomena beyond the group case have already been studied in several related settings. For \'etale groupoids, a version of property RD was introduced by Hou \cite{Hou2017}. For twists over \'etale groupoids, a twisted version was later developed by Weygandt \cite{Weygandt2024}, generalizing both Hou's groupoid setting and the earlier twisted group case appearing in the appendix by Chatterji to Mathai's work on twisted group algebras \cite{Mathai2006}. More recently, Fuller and Karmakar \cite{FullerKarmakar2024} used rapid decay methods in the study of Fourier coefficients and local approximation questions for reduced groupoid $C^*$-algebras. Further developments include the work of Austad, Ortega, and Palmstr{\o}m on property $RD_p$ and applications to $K$-theory \cite{AustadOrtegaPalmstrom2025}, as well as recent work of Stoiber on spectral continuity for \'etale groupoids with property RD \cite{Stoiber2025}.

The goal of the present paper is to extend this analytic framework to \emph{Fell bundles} over \'etale groupoids. Given a Fell bundle $E\to G$, we introduce Sobolev-type norms on compactly supported sections of $E$ and formulate a version of the Rapid Decay Property which controls the reduced cross-sectional norm in terms of these Sobolev norms. This provides a natural common generalization of the classical Rapid Decay Property for groups, the groupoid version of Hou, and the twisted groupoid setting studied in the literature.

One of the main consequences of RDP in this context is the existence of dense Schwartz-type $*$-subalgebras of the reduced cross-sectional algebra $C_r^*(E)$. These algebras behave well with respect to convolution and involution and provide analytic control over the reduced norm. Such smooth dense subalgebras play an important role in noncommutative geometry and in the study of approximation properties of operator algebras.

Our motivation for studying this property also comes from the notion of \emph{localizability} for Fell bundles introduced by Resende in the quantale framework \cite{Resende2017}. In that setting, localizability expresses a compatibility between the topology of the groupoid and the algebraic structure of a compatible completion of $C_c(E)$. From the analytic point of view adopted in this paper, a natural formulation of this phenomenon can be expressed directly in terms of the reduced $C^*$-algebra: if $U\subseteq G$ is open and $f\in C_r^*(E)$ satisfies $\supp(f)\subseteq U$, one would like to approximate $f$ in the reduced norm by compactly supported sections whose support is contained in $U$. This type of approximation property may be viewed as an analytic form of localizability for the reduced cross-sectional algebra.

A central theme of this paper is that the Rapid Decay Property provides a natural mechanism for obtaining such approximation results. Using positive definite multipliers and Schwartz-type regularization arguments, we show that, under suitable hypotheses, RDP implies a form of localizability for the reduced cross-sectional algebra of a Fell bundle. This connects the analytic theory developed here with the approximation and localizability phenomena studied in recent work of Pacheco \cite{PachecoThesis,Pacheco2024}.

Another central goal of this work is to understand the subtle interplay between geometric growth, dynamical systems, and the RDP. While Rapid Decay is a robust property for many discrete groups, its behavior in the groupoid setting is significantly more rigid. For instance, Weygandt \cite{Weygandt2024} recently showed that for principal \'etale groupoids, RDP is strictly equivalent to polynomial growth. We investigate this dichotomy in two main directions. 

First, we establish that any Fell bundle over a groupoid with polynomial growth naturally satisfies the RDP. Conversely, we completely characterize the RDP for the fundamental class of Deaconu-Renault groupoids. We prove that for covering maps of degree $d \geq 2$, the persistent branching of the dynamics forces the groupoid to have exponential growth, which completely obstructs the Rapid Decay Property. 

Second, we study Fell bundles associated with partial actions of discrete groups. We prove a reduction theorem showing that RDP for a Fell bundle over a partial transformation groupoid $G = \Gamma \ltimes_\theta X$ is equivalent to RDP for an associated Fell bundle over the acting group $\Gamma$. By realizing Deaconu-Renault groupoids as partial crossed products of free groups via Steinberg's model \cite{Steinberg2026}, we provide a striking application of this reduction: we exhibit natural dynamical systems where the acting group (a free group $F_d$) satisfies the RDP, but the exponential branching of the partial action space prevents the corresponding groupoid from inheriting it. This resolves a delicate boundary between group and groupoid Rapid Decay theories.

\medskip

The main contributions of this paper can be summarized as follows:

\begin{enumerate}
\item We formulate the Rapid Decay Property for Fell bundles over \'etale groupoids using Sobolev-type norms, and we use it to construct Schwartz-type dense $*$-subalgebras of the reduced cross-sectional algebra $C_r^*(E)$.

\item We show that using positive definite multiplier techniques and negative type functions, Rapid Decay implies analytic forms of localizability for the reduced cross-sectional algebra.

\item We establish that every Fell bundles over a groupoid with polynomial growth satisfies the RDP, and we prove that persistent branching in Deaconu-Renault groupoids yields exponential growth, completely obstructing the property.

\item We prove a reduction theorem for partial actions, showing that RDP for a partial transformation groupoid is equivalent to RDP for an associated Fell bundle over the acting group, and we apply this to Steinberg's free group model to clarify the limitations of RDP in non-injective dynamics.
\end{enumerate}

\medskip

The paper is organized as follows.

In Section~\ref{sec:preliminaries} we recall the necessary background on Fell bundles, \'etale groupoids, and localizability.

In Section~\ref{sec:RD-Fell-bundles} we introduce the Rapid Decay Property for Fell bundles and establish basic analytic estimates, including the construction of Schwartz-type subalgebras.

In Section~\ref{sec:Fell-bundles-poly-growth} we explore the deep connection between RDP and the geometric growth of the groupoid. Extending classical results, we show that Fell bundles over groupoids with polynomial growth satisfy the RDP. Conversely, we completely characterize the RDP for the fundamental class of Deaconu-Renault groupoids, proving that persistent branching obstructs the property (Proposition~\ref{prop:ObstructionRDP}).

In Section~\ref{sec:continuity-locality} we establish fundamental continuity and locality properties of the $l^2$-norms. These estimates provide the essential analytic tools required for our subsequent approximation and localizability results.

In Section~\ref{sec:partial-actions-reduction} we study the behavior of RDP for Fell bundles associated with partial actions. We show that RDP for a partial transformation groupoid reduces to the group case, and we apply this to Steinberg's free group model to clarify the limitations of RDP in non-injective dynamics.

In Section~\ref{sec:group-actions} we analyze Fell bundles arising from global actions of groups on $C^*$-algebras. We establish necessary conditions for RDP in this setting and prove that Rapid Decay holds for trivial actions on commutative and finite-dimensional algebras.

Finally, in Section~\ref{sec:localizability} we present our main analytic application. We show how Rapid Decay, combined with positive definite multiplier techniques and the Haagerup property, yields local approximation results. In particular, we prove that under suitable exactness and RDP assumptions, sections of the reduced $C^*$-algebra supported in an open subset can be approximated by continuous sections with compact support in the same set, establishing an analytic form of localizability.

\subsection*{Acknowledgements}
The first author was supported by CNPq and FAPESC.

\section{Preliminaries}\label{sec:preliminaries}

Throughout the paper, $G$ denotes a locally compact Hausdorff \'etale groupoid
with unit space $G^{(0)}$.
We write
\[
s,r\colon G\to G^{(0)}
\]
for the source and range maps.
For $x\in G^{(0)}$, we use the standard notation
\[
G_x:=s^{-1}(x),
\qquad
G^x:=r^{-1}(x).
\]

Let $E=\{E_\gamma\}_{\gamma\in G}$ be a saturated Fell bundle over $G$.
We denote by $C_c(E)$ the $*$-algebra of compactly supported continuous
sections of $E$, endowed with the usual convolution and involution, and by
\[
C_r^*(E)
\]
its reduced cross-sectional $C^*$-algebra.
If $U\subseteq G$ is open, we write $E|_U$ for the restricted bundle over $U$
and $C_c(E|_U)$ for the compactly supported continuous sections of this
restriction.

We shall freely view $C_c(E)$ as a subspace of the space of sections of the
bundle, and for $f\in C_c(E)$ we write
\[
\supp(f):=\{\gamma\in G: f(\gamma)\neq 0\}.
\]
More generally, whenever $C_r^*(E)$ is realized as a subspace of $C_0(E)$, we
use the same notation for the support of elements of $C_r^*(E)$.

A \emph{length function} on $G$ is a continuous map
\[
L\colon G\to [0,\infty)
\]
such that:
\begin{enumerate}[label=\textup{(\roman*)}]
\item $L(\gamma^{-1})=L(\gamma)$ for all $\gamma\in G$;
\item $L(x)=0$ for all $x\in G^{(0)}$;
\item $L(\gamma\eta)\leq L(\gamma)+L(\eta)$ whenever $s(\gamma)=r(\eta)$.
\end{enumerate}

When $\Gamma$ is a discrete group, we write $\ell\colon \Gamma\to [0,\infty)$
for a length function on $\Gamma$, reserving the notation $L$ for length
functions on groupoids.


\section{Rapid Decay for Fell Bundles}\label{sec:RD-Fell-bundles}

Let $G$ be a locally compact Hausdorff \'etale groupoid with unit space
$G^{(0)}$, and let $L\colon G\to [0,\infty)$ be a length function, that is,
a continuous map such that
\begin{enumerate}[label=\textup{(\roman*)}]
\item $L(\gamma^{-1})=L(\gamma)$ for all $\gamma\in G$;
\item $L(x)=0$ for all $x\in G^{(0)}$;
\item $L(\gamma\eta)\leq L(\gamma)+L(\eta)$ whenever $(\gamma,\eta)\in G^{(2)}$.
\end{enumerate}
For $x\in G^{(0)}$ we write
\[
G_x=s^{-1}(x),
\qquad
G^x=r^{-1}(x).
\]

We first recall the scalar-valued version.
For $f\in C_c(G)$ and $p\in \mathbb Z_+$ a non-negative integer, define
\begin{align*}
\|f\|_{2,p,s,L}
&:=
\sup_{x\in G^{(0)}}
\Big(
\sum_{\gamma\in G_x}|f(\gamma)|^2(1+L(\gamma))^{2p}
\Big)^{1/2},\\
\|f\|_{2,p,r,L}
&:=
\sup_{x\in G^{(0)}}
\Big(
\sum_{\gamma\in G^x}|f(\gamma)|^2(1+L(\gamma))^{2p}
\Big)^{1/2},
\end{align*}
and
\[
\|f\|_{2,p,L}:=\max\{\|f\|_{2,p,s,L},\|f\|_{2,p,r,L}\}.
\]
We say that $G$ has the \emph{Rapid Decay Property} with respect to $L$ if there
exist constants $C>0$ and $p> 0$(positive integer) such that
\[
\|f\|_r\leq C\|f\|_{2,p,L}
\qquad\text{for all }f\in C_c(G).
\]

We now extend this definition to Fell bundles.

Let $E=\{E_\gamma\}_{\gamma\in G}$ be a Fell bundle over $G$.
We do \emph{not} assume that $E$ is saturated.
Write $C_c(E)$ for the $*$-algebra of compactly supported continuous sections,
and $C_r^*(E)$ for its reduced cross-sectional $C^*$-algebra.

For $f\in C_c(E)$ define
\begin{align*}
\|f\|_{2,p,s,L}
&:=
\sup_{x\in G^{(0)}}
\Big\|
\sum_{\gamma\in G_x} f(\gamma)^*f(\gamma)(1+L(\gamma))^{2p}
\Big\|^{1/2},\\
\|f\|_{2,p,r,L}
&:=
\sup_{x\in G^{(0)}}
\Big\|
\sum_{\gamma\in G^x} f(\gamma)f(\gamma)^*(1+L(\gamma))^{2p}
\Big\|^{1/2},
\end{align*}
and
\[
\|f\|_{2,p,L}:=\max\{\|f\|_{2,p,s,L},\|f\|_{2,p,r,L}\}.
\]

For $p=0$ we also write
\[
\|f\|_{2,s}:=\|f\|_{2,0,s,L},
\qquad
\|f\|_{2,r}:=\|f\|_{2,0,r,L},
\]
\[
\|f\|_{II}:=\|f\|_{2,0,L}=\max\{\|f\|_{2,s},\|f\|_{2,r}\}.
\]
Because $L$ is symmetric, one has
\[
\|f\|_{2,p,r,L}=\|f^*\|_{2,p,s,L},
\qquad
\|f\|_{2,p,L}=\|f^*\|_{2,p,L}.
\]

If $p\geq 0$, let $H^{2,p,L}(E)$ denote the completion of $C_c(E)$ with respect
to the norm $\|\cdot\|_{2,p,L}$.
We then define the associated Schwartz-type space by
\[
H^{2,L}(E):=\bigcap_{p\in \mathbb Z_+} H^{2,p,L}(E)\cap C_0(G,E),
\]
equipped with the Fr\'echet topology given by the seminorms
$\|\cdot\|_{2,p,L}$, $p\geq 0$.

\begin{definition}
We say that the Fell bundle $E$ has the \emph{Rapid Decay Property} with
respect to $L$ if there exist constants $C>0$ and $p\in \mathbb Z_+$ such that
\[
\|f\|_r\leq C\|f\|_{2,p,L}
\qquad\text{for all }f\in C_c(E).
\]
\end{definition}

\begin{remark}
When $E=\bC\times G$ is the trivial line bundle, this recovers the usual Rapid
Decay Property for the groupoid $G$.
More generally, for any Fell line bundle over $G$, the above definition agrees
with the twisted groupoid versions appearing in the literature.
\end{remark}

\begin{remark}
For later use, let us also record the $I$-norm
\[
\|f\|_{I}
:=
\max\Big\{
\sup_{x\in G^{(0)}}\sum_{\gamma\in G^x}\|f(\gamma)\|,
\;
\sup_{x\in G^{(0)}}\sum_{\gamma\in G_x}\|f(\gamma)\|
\Big\}.
\]
For every $f\in C_c(E)$ one has
\begin{equation}\label{eq:norm-inequalities}
\|f\|_\infty\leq \|f\|_{II}\leq \|f\|_r\leq \|f\|_{I}.
\end{equation}
\end{remark}

\begin{proposition}\label{prop:RD-gives-continuous-inclusion}
Assume that $E$ has the Rapid Decay Property with respect to $L$.
Then the identity map on $C_c(E)$ extends uniquely to a continuous linear map
\[
\iota\colon H^{2,L}(E)\longrightarrow C_r^*(E).
\]
In particular, we may regard $H^{2,L}(E)$ as a dense subspace of $C_r^*(E)$.
\end{proposition}

\begin{proof}
By assumption there exist constants $C>0$ and $q> 0$ such that
\[
\|f\|_r\leq C\|f\|_{2,q,L}
\qquad\text{for all }f\in C_c(E).
\]
Thus the identity map on $C_c(E)$ is continuous from the normed space
$(C_c(E),\|\cdot\|_{2,q,L})$ into $C_r^*(E)$, and therefore extends uniquely to
a continuous map
\[
H^{2,q,L}(E)\to C_r^*(E).
\]
Restricting this map to $H^{2,L}(E)\subseteq H^{2,q,L}(E)$ gives the desired
continuous linear map
\[
\iota\colon H^{2,L}(E)\to C_r^*(E).
\]
Its image contains $C_c(E)$, hence is dense in $C_r^*(E)$.
\end{proof}

\begin{lemma}\label{lem:weighted-convolution-estimate}
Assume that $E$ has the Rapid Decay Property with respect to $L$, witnessed by
constants $C>0$ and $q> 0$.
Then for every $p> 0$ and every $f,g\in C_c(E)$ one has
\begin{align*}
\|f*g\|_{2,p,s,L}
&\leq C\,\|f\|_{2,p+q,L}\,\|g\|_{2,p,s,L},\\
\|f*g\|_{2,p,r,L}
&\leq C\,\|g\|_{2,p+q,L}\,\|f\|_{2,p,r,L}.
\end{align*}
In particular,
\[
\|f*g\|_{2,p,L}
\leq
C\max\big\{
\|f\|_{2,p+q,L}\|g\|_{2,p,L},
\|g\|_{2,p+q,L}\|f\|_{2,p,L}
\big\}.
\]
\end{lemma}

\begin{proof}
Fix $p\geq 0$, and write for $\gamma\in G$,
\[
f_p(\gamma):=f(\gamma)(1+L(\gamma))^p,
\qquad
g_p(\gamma):=g(\gamma)(1+L(\gamma))^p.
\]
Since
\[
1+L(\gamma\eta)\leq (1+L(\gamma))(1+L(\eta))\quad\mbox{for all }\gamma,\eta\in G,
\]
we obtain the pointwise estimate
\[
\|(f*g)(\zeta)\|(1+L(\zeta))^p
\leq
\sum_{\gamma\eta=\zeta}\|f_p(\gamma)\|\,\|g_p(\eta)\|.
\]
Hence
\[
\|f*g\|_{2,p,s,L}\leq \|f_p*g_p\|_{2,s}.
\]
Now left convolution by $f_p$ on the source-side Hilbert module is bounded by
$\|f_p\|_r$, so
\[
\|f_p*g_p\|_{2,s}\leq \|f_p\|_r\,\|g_p\|_{2,s}.
\]
Using Rapid Decay for $f_p$, we get
\[
\|f_p\|_r\leq C\|f_p\|_{2,q,L}=C\|f\|_{2,p+q,L},
\]
and therefore
\[
\|f*g\|_{2,p,s,L}\leq C\,\|f\|_{2,p+q,L}\,\|g\|_{2,p,s,L}.
\]

The range estimate follows by applying the source estimate to
\[
(f*g)^*=g^**f^*,
\]
namely
\begin{align*}
\|f*g\|_{2,p,r,L}
&=\|(f*g)^*\|_{2,p,s,L}\\
&=\|g^**f^*\|_{2,p,s,L}\\
&\leq C\,\|g^*\|_{2,p+q,L}\,\|f^*\|_{2,p,s,L}\\
&=C\,\|g\|_{2,p+q,L}\,\|f\|_{2,p,r,L}.
\end{align*}
\end{proof}

\begin{theorem}\label{thm:schwartz-subalgebra}
Assume that $E$ has the Rapid Decay Property with respect to $L$.
Then $H^{2,L}(E)$ is a dense involutive Fr\'echet subalgebra of $C_r^*(E)$.
\end{theorem}

\begin{proof}
Density follows from Proposition~\ref{prop:RD-gives-continuous-inclusion},
since $C_c(E)$ is dense in both $H^{2,L}(E)$ and $C_r^*(E)$.

The involution is continuous because
\[
\|f^*\|_{2,p,L}=\|f\|_{2,p,L}
\qquad\text{for all }p\geq 0.
\]

Let $f,g\in H^{2,L}(E)$, and choose sequences $(f_n)$ and $(g_n)$ in $C_c(E)$
such that
\[
f_n\to f,\qquad g_n\to g
\]
in the Fr\'echet topology of $H^{2,L}(E)$.
Fix $p\geq 0$.
By Lemma~\ref{lem:weighted-convolution-estimate},
\[
\|f_n*g_n-f_m*g_m\|_{2,p,L}
\]
is bounded by a linear combination of
\[
\|f_n-f_m\|_{2,p+q,L}\|g_n\|_{2,p,L},
\qquad
\|g_n-g_m\|_{2,p+q,L}\|f_m\|_{2,p,L}.
\]
Since $(f_n)$ and $(g_n)$ are Cauchy in every seminorm $\|\cdot\|_{2,r,L}$,
it follows that $(f_n*g_n)$ is Cauchy in each seminorm $\|\cdot\|_{2,p,L}$.
Hence it converges in $H^{2,L}(E)$ to some element, denoted by $f*g$.

Because the inclusion
\[
H^{2,L}(E)\hookrightarrow C_r^*(E)
\]
is continuous, this limit agrees with the convolution product in $C_r^*(E)$.
Thus $H^{2,L}(E)$ is closed under convolution.
\end{proof}

\begin{remark}
In the sequel, the main role of $H^{2,L}(E)$ is as a regularizing subalgebra:
under suitable multiplier assumptions, one shows that certain elements of
$C_r^*(E)$ belong to $H^{2,L}(E)$, and hence can be approximated locally by
compactly supported sections.
\end{remark}


Here is a general permanence property of the Rapid Decay Property:

\begin{proposition}\label{prop:RD-open-subgroupoid}
Let $G$ be an \'etale groupoid equipped with a length function $L$, and let
$E\to G$ be a Fell bundle.
Suppose that $E$ has the Rapid Decay Property with respect to $L$.

Let $H\subseteq G$ be an open subgroupoid and denote by
$E|_H \to H$ the restricted Fell bundle.
Then $E|_H$ also has the Rapid Decay Property with respect to the restricted
length function $L|_H$.
\end{proposition}

\begin{proof}
Let $f\in C_c(E|_H)$.
Viewing $f$ as a compactly supported section of $E$ by extending it by $0$
outside $H$, we obtain an element of $C_c(E)$.

Since $H$ is a subgroupoid, the convolution and involution computed in
$C_c(E|_H)$ coincide with those computed in $C_c(E)$.
Moreover, the reduced regular representation of $E|_H$ at a unit
$x\in H^{(0)}$ is naturally identified with the restriction of the reduced
regular representation of $E$ at $x$.
Hence
\[
\|f\|_{r,H} \le \|f\|_{r,G}.
\]

For the Sobolev-type norms we have
\[
\|f\|_{2,p,s,L|_H}
=
\sup_{x\in H^{(0)}}
\Big(\sum_{\gamma\in H_x}
\|f(\gamma)\|^2(1+L(\gamma))^{2p}\Big)^{1/2}
\le
\|f\|_{2,p,s,L},
\]
since $H_x\subseteq G_x$.
Similarly,
\[
\|f\|_{2,p,r,L|_H} \le \|f\|_{2,p,r,L}.
\]
Therefore
\[
\|f\|_{2,p,L|_H}\le \|f\|_{2,p,L}.
\]

Since $E$ has Rapid Decay with respect to $L$, there exist constants
$C>0$ and $p>0$ such that
\[
\|f\|_{r,G}\le C\|f\|_{2,p,L}.
\]
Combining the above inequalities yields
\[
\|f\|_{r,H}
\le
\|f\|_{r,G}
\le
C\|f\|_{2,p,L}
\le
C\|f\|_{2,p,L|_H}.
\]
Hence $E|_H$ has the Rapid Decay Property with respect to $L|_H$.
\end{proof}

\begin{remark}
    It would also be natural to study stability properties of the Rapid Decay
Property under constructions such as products of groupoids and exterior
tensor products of Fell bundles. For instance, the Rapid Decay Property
passes to products $G\times H$ whenever $G$ has RD and $H$ is compact,
see~\cite[Proposition~4.2]{Weygandt2024}.
However, extending this type of stability result to Fell bundles appears
to require a careful treatment of exterior tensor products of Fell bundles
over groupoids, which we leave for future work.
\end{remark}

\section{Fell Bundles over Groupoids with polynomial growth}\label{sec:Fell-bundles-poly-growth}

In this section, we show that the Rapid Decay Property is satisfied by a very broad class of Fell bundles, namely those whose underlying groupoid exhibits polynomial growth. This covers, for instance, all Fell bundles over finitely generated abelian groups (such as $\mathbb{Z}^n$) and transformation groupoids arising from partial actions of polynomial growth groups.

Let us first recall the definition of polynomial growth for \'etale groupoids, which was investigated by Austad, Ortega, and Palmstr\o m in \cite[Definition~3.11]{AustadOrtegaPalmstrom2025}.

\begin{definition}[{\cite[Definition~3.11]{AustadOrtegaPalmstrom2025}}]
    Let $G$ be an \'etale groupoid equipped with a length function $L$. We say that $G$ has \emph{polynomial growth} with respect to $L$ if there exist constants $c > 0$ and $d > 0$ such that for all $x \in G^{(0)}$ and all $n \geq 0$,
    \begin{align*}
        |B_{G_x}(n)| \leq c(1+n)^d \,,
    \end{align*}
    where $B_{G_x}(n) = \{g \in G_x : L(g) \leq n\}$ is the closed ball of radius $n$ in the fiber $G_x$.
\end{definition}

\begin{example}[Partial transformation groupoids]
    Let $\Gamma$ be a discrete group with polynomial growth with respect to a length function $\ell$. This means there exist constants $c > 0$ and $d \geq 1$ such that $|B_\Gamma(n)| \leq c(1+n)^d$ for all $n \geq 0$, where $B_\Gamma(n) = \{ \gamma \in \Gamma : \ell(\gamma) \leq n \}$.
    
    Consider a partial action $\theta$ of $\Gamma$ on a locally compact Hausdorff space $X$. The associated partial transformation groupoid $G = \Gamma \ltimes_\theta X$ is an \'etale groupoid with unit space $G^{(0)} \cong X$. We can define a natural length function $L$ on $G$ by setting $L(\gamma, x) = \ell(\gamma)$ for every $(\gamma, x) \in G$. 
    
    For any $x \in X$, the source fiber is $G_x = \{ (\gamma, x) : x \in U_{\gamma^{-1}} \}$. Since the map $(\gamma, x) \mapsto \gamma$ is an injection from $G_x$ into $\Gamma$, we have
    \begin{align*}
        |B_{G_x}(n)| &= |\{ (\gamma, x) \in G_x : L(\gamma, x) \leq n \}| \\
        &\leq |\{ \gamma \in \Gamma : \ell(\gamma) \leq n \}| = |B_\Gamma(n)| \leq c(1+n)^d
    \end{align*}
    for all $x \in X$ and $n \geq 0$. Thus, $G$ has polynomial growth with respect to $L$. 
    
    In particular, by Theorem \ref{teo:pol-growth}, it will follow that any Fell bundle over a partial transformation groupoid $G = \Gamma \ltimes_\theta X$ has the RDP whenever $\Gamma$ is a group with polynomial growth (e.g., if $\Gamma$ is a finitely generated abelian or nilpotent group).
\end{example}

\begin{example}[AF groupoids and adic dynamics]
    The case of $\Gamma = \mathbb{Z}$ in the previous example is already rich enough to encompass the entire class of AF (Approximately Finite) groupoids. As established by the foundational works of Renault \cite{Renault1980}, Herman, Putnam, Skau \cite{Herman1992}, and Exel \cite{Exel1993}, any AF groupoid $G$---which is classically modeled as the tail equivalence relation on the path space $X$ of a Bratteli diagram---can be realized as a partial transformation groupoid:
    $$G \cong \mathbb{Z} \ltimes_\theta X,$$
    where $\theta$ is the adic (or Vershik) transformation induced by an ordering on the diagram. 
    
    This structural realization equips $G$ with a canonical ``adic'' length function inherited from $\mathbb{Z}$, given by $L(n, x) = |n|$. Under this specific length function, the groupoid inherits the linear growth of $\mathbb{Z}$. Thus, while a naive combinatorial length function on the Bratteli diagram might exhibit exponential growth, the dynamical realization guarantees the existence of a natural length function of polynomial growth (degree 1). It then follows from Theorem~\ref{teo:pol-growth} below that any Fell bundle over an AF groupoid satisfies the RDP with respect to this adic length function.
\end{example}

This next theorem provides a large class of examples of Fell bundles which admit rapid decay. The proof is an adaptation of the arguments found in \cite[Proposition~3.14]{AustadOrtegaPalmstrom2025} to the context of Fell bundles.

\begin{theorem}\label{teo:pol-growth}
    Let $G$ be an \'etale groupoid with polynomial growth with respect to a length function $L$. Then any Fell bundle $E$ over $G$ has RDP with respect to $L$.
\end{theorem}
\begin{proof}
    Since $G$ has polynomial growth, there exist a constant $c > 0$ and a positive integer $t \geq 1$ such that
    \begin{align*}
        \sup_{x \in G^{(0)}}|B_{G_x}(m)| \leq c (1+m)^t 
    \end{align*}
    for all $m \geq 0$. Fix $x \in G^{(0)}$ and let $k=2+t$. We estimate the sum:
    \begin{align*}
        \sum_{g \in G_x}(1+L(g))^{-2k} &= \sum_{n=0}^{\infty} \sum_{\substack{g \in G_x \\ n \leq L(g) < n+1}}(1+L(g))^{-2k}\\
        & \leq \sum_{n=0}^{\infty}|B_{G_x}(n+1)|(1+n)^{-2k}\\
        & \leq  c \sum_{n=0}^{\infty} (2+n)^t (1+n)^{-2k}\\
        & \leq 2^t c \sum_{n=0}^{\infty} (1+n)^t (1+n)^{-2(2+t)}\\
        & = 2^t c \sum_{n=0}^{\infty} (1+n)^{-(t+4)} \leq 2^t c \sum_{n=0}^{\infty} (1+n)^{-4} =: c_1 \,.
    \end{align*}
    Now we use the following inequality, which is a particular case of \cite[Remark~5.5]{BHM:universal-propII},
    \begin{align*}
        \|f\|_r^2 \leq \sup_{x \in G^{(0)}}\left\|\sum_{g \in G_x} (f(g)^{\ast}f(g))^{1/2}\right\| \cdot \sup_{x \in G^{(0)}}\left\|\sum_{g \in G^x} (f(g)f(g)^{\ast})^{1/2}\right\|,
    \end{align*}
    for all $f \in C_c(E)$.
    Recall that in any $C^*$-algebra $A$, for any finite sequence of positive elements $a_g \in A^+$ and scalars $\lambda_g \geq 0$, the generalized Cauchy-Schwarz inequality holds:
    \begin{align}\label{impineq}
        \sum_g \lambda_g a_g \leq \left(\sum_g \lambda_g^2\right)^{\frac{1}{2}} \left(\sum_g a_g^2\right)^{\frac{1}{2}}\,.
    \end{align}
    Applying \eqref{impineq} inside the $C^*$-algebra $E_x$ with $a_g = (f(g)^{\ast}f(g))^{1/2}(1+L(g))^k$ and $\lambda_g = (1+L(g))^{-k}$, we obtain:
    \begin{align*}
        \sum_{g \in G_x} (f(g)^{\ast}f(g))^{1/2} &=\sum_{g \in G_x} (f(g)^{\ast}f(g))^{1/2} (1+L(g))^k (1+L(g))^{-k}\\
        & \leq \left(\sum_{g \in G_x} (1+L(g))^{-2k}\right)^{\frac{1}{2}} \left(\sum_{g \in G_x} f(g)^{\ast} f(g) (1+L(g))^{2k} \right)^{\frac{1}{2}}\\
        & \leq \sqrt{c_1} \left(\sum_{g \in G_x} f(g)^{\ast} f(g) (1+L(g))^{2k} \right)^{\frac{1}{2}}\,.
    \end{align*}
    Taking the norm and the supremum over $x \in G^{(0)}$, we get
    \begin{multline*}
        \sup_{x \in G^{(0)}} \left\|\sum_{g \in G_x} (f(g)^{\ast}f(g))^{1/2} \right\| \leq \\
         \sqrt{c_1} \sup_{x \in G^{(0)}} \left\|\left(\sum_{g \in G_x} f(g)^{\ast} f(g) (1+L(g))^{2k} \right)^{\frac{1}{2}}\right\| \leq \sqrt{c_1} \|f\|_{2,k}\,.
    \end{multline*}
    By a completely analogous argument (using the inversion map on $G$ to sum over $G^x$), we also get:
    \begin{multline*}
        \sup_{x \in G^{(0)}} \left\|\sum_{g \in G^x} (f(g)f(g)^{\ast})^{1/2} \right\| \leq \\
        \sqrt{c_1} \sup_{x \in G^{(0)}} \left\|\left(\sum_{g \in G^x} f(g)f(g)^{\ast} (1+L(g))^{2k} \right)^{\frac{1}{2}}\right\| \leq \sqrt{c_1} \|f\|_{2,k}\,.
     \end{multline*}
    Hence 
    \begin{align*}
        \|f\|_r^2 \leq c_1 \|f\|_{2,k}^2\,,
    \end{align*}
    which implies $\|f\|_r \leq \sqrt{c_1} \|f\|_{2,k}$. Thus the result follows.
\end{proof}

As an immediate consequence of the previous theorem, we recover the fact that any Fell bundle over a compact groupoid has the Rapid Decay Property.

\begin{corollary}
Let $G$ be a compact Hausdorff \'etale groupoid, and let $L\colon G\to [0,\infty)$ be a continuous length function. Then any Fell bundle $E \to G$ has the Rapid Decay Property with respect to $L$.
\end{corollary}

\begin{proof}
Since $G$ is compact and \'etale, there exists a finite family of open bisections covering $G$. Hence, there is a uniform bound $n \in \mathbb{N}$ such that $|G_x|\le n$ for all $x\in G^{(0)}$. This implies that $G$ has polynomial growth of degree $0$ with respect to any length function. The result then follows directly from Theorem~\ref{teo:pol-growth}.
\end{proof}

\begin{example}[Deaconu-Renault groupoids and exponential growth]
    It is important to emphasize that polynomial growth is not a universal feature of dynamical groupoids. Consider the Deaconu-Renault groupoid $G$ associated with a local homeomorphism $T \colon X \to X$ on a locally compact Hausdorff space $X$, defined by
    \[
    G = \{ (x, m-n, y) \in X \times \mathbb{Z} \times X : T^m(x) = T^n(y), \, m,n \in \mathbb{N}_0 \},
    \]
    where $s(x, k, y) = y$ and $r(x, k, y) = x$. We equip $G$ with the natural length function $L(x, k, y) = \min \{ m+n : k = m-n, \, T^m(x) = T^n(y) \}$. 
    
    To analyze the growth of the fibers, suppose that $T$ is a $d$-to-1 covering map with $d \geq 2$. This naturally occurs, for instance, when $T(z)=z^d$ on the circle $S^1$, or when $T$ is the shift map on the totally disconnected full shift space $X = \{1, \dots, d\}^{\mathbb{N}}$. Considering the ``backward'' dynamics where $m=0$, each point $x$ has exactly $d^n$ distinct pre-images under $T^n$. Each such pre-image $y$ corresponds to an element $(x, -n, y) \in G_x$ of length $L = n$. Thus, the size of the ball of radius $N$ in the fiber satisfies
    \[
    |B_{G_x}(N)| \geq \sum_{n=0}^N d^n = \frac{d^{N+1}-1}{d-1}.
    \]
    Since $d \geq 2$, this grows exponentially. Consequently, $G$ does not have polynomial growth. While $G$ naturally has polynomial growth when $T$ is a global homeomorphism ($G \cong \mathbb{Z} \ltimes_T X$), the presence of persistent branching leads to exponential growth, showing that Theorem \ref{teo:pol-growth} does not apply to groupoids associated with Cuntz-like algebras.
\end{example}

\begin{remark}[Polynomial growth with non-injective dynamics]
    It is worth noting that if the branching of the pre-images is not persistent, a Deaconu-Renault groupoid can have polynomial growth even when the local homeomorphism $T$ is not injective. 
    
    Consider the discrete space $X = \{a, b_0, b_1, b_2, \dots\}$. Define a map $T \colon X \to X$ by $T(a) = a$, $T(b_0) = a$, and $T(b_n) = b_{n-1}$ for all $n \ge 1$. Since $X$ is discrete, $T$ is trivially a local homeomorphism. Furthermore, $T$ is surjective but not injective, because $T(a) = T(b_0) = a$.
    
    However, the branching in the pre-image tree of $T$ does not multiply. The pre-images of $a$ under $T^N$ form the set $\{a, b_0, b_1, \dots, b_{N-1}\}$, which has exactly $N+1$ elements. For any other point $b_k$, the number of pre-images is even smaller (in fact, it is exactly $1$). Consequently, the maximum number of pre-images of any point under $T^N$ grows only linearly with $N$. 
    
    When we construct the Deaconu-Renault groupoid $G$ associated with $T$, the balls in the fibers will only exhibit linear growth, meaning $G$ has polynomial growth of degree 1. This happens despite $T$ failing to be a global homeomorphism. This example highlights precisely why the assumption of persistent branching (such as $T$ being a $d$-to-1 covering map) is essential to obtain exponential growth and obstruct the Rapid Decay Property.
\end{remark}

\begin{proposition}[Obstruction to RDP for branching dynamics]\label{prop:ObstructionRDP}
    Let $X$ be a locally compact Hausdorff space and $T \colon X \to X$ a $d$-to-1 covering map with $d \geq 2$. Let $G$ be the associated Deaconu-Renault groupoid equipped with the natural length function $L(x,k,y) = \min\{m+n : k = m-n, \, T^m(x)=T^n(y)\}$. Then $G$ does not have the Rapid Decay Property.
\end{proposition}

\begin{proof}
    Consider the canonical continuous cocycle $c \colon G \to \mathbb{Z}$ given by $c(x, k, y) = k$. Its kernel $R = c^{-1}(0)$ is an open and closed subgroupoid of $G$. The elements of $R$ are exactly the triples $(x, 0, y)$ such that $T^n(x) = T^n(y)$ for some $n \in \mathbb{N}_0$. Thus, $R$ is an equivalence relation, which means it is a principal groupoid.
    
    Suppose for a contradiction that $G$ has the Rapid Decay Property with respect to $L$. Since $R$ is an open subgroupoid of $G$, it must inherit the Rapid Decay Property with respect to the restricted length function $L|_R$. 
    
    Let us analyze the growth of $R$. The restriction of $L$ to $R$ is given by
    \[
    L(x, 0, y) = \min \{ 2n : T^n(x) = T^n(y) \}.
    \]
    Fix a point $x \in X$ and let $N \in \mathbb{N}$. The closed ball of radius $2N$ in the source fiber $R_x$ is
    \[
    B_{R_x}(2N) = \{ (x, 0, y) \in R : L(x, 0, y) \leq 2N \}.
    \]
    Notice that for any $y \in X$ satisfying $T^N(y) = T^N(x)$, the triple $(x, 0, y)$ belongs to $R$ and has length $L(x, 0, y) \leq 2N$. Let $z = T^N(x)$. Because $T$ is a covering map of degree $d$, the $N$-th iterate $T^N$ is a covering map of degree $d^N$. Therefore, the point $z$ has exactly $d^N$ distinct pre-images in $X$. This means there are $d^N$ choices for $y$, all of which are contained in $B_{R_x}(2N)$. 
    
    Consequently, $|B_{R_x}(2N)| \geq d^N$, which shows that $R$ has exponential growth. 
    
    However, Weygandt \cite{Weygandt2024} proved that for principal Hausdorff \'etale groupoids, the Rapid Decay Property is equivalent to polynomial growth. Since $R$ has exponential growth, it cannot satisfy the RDP. This contradicts our assumption that $G$ has RDP, completing the proof.
\end{proof}

\begin{example}[Directed graphs and Cuntz-Krieger groupoids]
    The dichotomy of growth for Deaconu-Renault groupoids is particularly striking in the context of directed graphs. Let $E$ be a finite directed graph with no sinks, and let $X = E^\infty$ be the locally compact Hausdorff space of infinite paths. The shift map $\sigma \colon E^\infty \to E^\infty$, given by $\sigma(e_1 e_2 e_3 \dots) = e_2 e_3 \dots$, is a surjective local homeomorphism. The associated Deaconu-Renault groupoid $G_E$ is the path groupoid whose reduced $C^*$-algebra is exactly the graph $C^*$-algebra $C^*(E)$.

    The Rapid Decay Property for $G_E$ is highly restricted and depends entirely on the branching structure of the graph $E$:
    \begin{enumerate}
    \item \emph{No persistent branching (Polynomial growth):} Suppose $E$ consists of a single vertex and a single loop. Then $E^\infty$ is a single point, $\sigma$ is the identity map, and $G_E \cong \mathbb{Z}$. Since $\mathbb{Z}$ has polynomial growth, $G_E$ has the RDP. More generally, if $E$ is a finite graph where no cycle has an exit, the classical structure theory of graph algebras dictates that $C^*(E)$ is Morita equivalent to a commutative $C^*$-algebra (specifically, a finite direct sum of copies of $C(\mathbb{T})$). Dynamically, this means the pre-image trees under the shift $\sigma$ do not multiply persistently. The fibers of the groupoid exhibit bounded (hence polynomial) growth, allowing the RDP to hold.        
        \item \emph{Persistent branching (Exponential growth):} Suppose $E$ contains a vertex that emits multiple edges within a strongly connected component. A canonical example is the graph with one vertex and $d \ge 2$ loops. Here, $C^*(E)$ is the Cuntz algebra $\mathcal{O}_d$, and the shift map $\sigma$ is exactly a $d$-to-1 covering map on $E^\infty$. By Proposition~\ref{prop:ObstructionRDP}, the persistent branching of the paths forces the principal subgroupoid $R = \ker(c)$ to have exponential growth. Consequently, $G_E$ completely fails the Rapid Decay Property. 
    \end{enumerate}
    This application illustrates that within the rich class of graph $C^*$-algebras, the Rapid Decay Property essentially separates ``abelian-like'' dynamics from purely infinite (Cuntz-Krieger) dynamics.
\end{example}

\begin{remark}
Recent work by Austad, Ortega, and Palmstrøm \cite{AustadOrtegaPalmstrom2025} constructs specific shifts of ``infinite type'', which admit a length function leading to strong subexponential growth. This highlights that the RDP for Renault-Deaconu groupoids is not a uniform property; rather, it is governed by the topological entropy and the specific branching pattern of the shift map. While our work focuses on the exponential regime—a characteristic feature of finite-type dynamics—their examples explore the critical boundary of subexponential growth, where the RDP may still be attainable.
\end{remark}

\begin{remark}
The exponential growth exhibited by the $d$-to-1 Deaconu-Renault groupoid is not incidental. As noted in \cite{AustadOrtegaPalmstrom2025}, for minimal and effective groupoids with compact unit space, the pure infiniteness of $C^*_r(G)$ forces exponential growth of the groupoid (with respect to any locally bounded length function). Since the $d$-to-1 dynamics typically give rise to purely infinite algebras, the resulting exponential growth acts as a barrier to the RDP, reinforcing the link between the pure infiniteness of the algebra and the geometric growth of the underlying groupoid.
\end{remark}

\section{Continuity and locality of the $\ell^2$-norm}\label{sec:continuity-locality}

In this section we establish two elementary facts that will be used later in the
proof of the local approximation theorem. First, the fiberwise sums defining the
$\ell^2$-norm vary continuously over the unit space. Second, if an $\ell^2$-section
is supported in an open subset, then it can be approximated in $\ell^2$-norm by
compactly supported sections inside that open set.

Let $E\to G$ be a Fell bundle over a locally compact Hausdorff \'etale groupoid $G$.

\subsection{The source and range $\ell^2$-norms}

For $f\in C_c(E)$, we define
\[
\|f\|_{2,s}
:=
\sup_{x\in G^{(0)}}
\Big\|
\sum_{\gamma\in G_x} f(\gamma)^*f(\gamma)
\Big\|^{1/2},
\qquad
\|f\|_{2,r}
:=
\sup_{x\in G^{(0)}}
\Big\|
\sum_{\gamma\in G^x} f(\gamma)f(\gamma)^*
\Big\|^{1/2}.
\]
Thus
\[
\|f\|_{II}=\max\{\|f\|_{2,s},\|f\|_{2,r}\},
\]
in agreement with Section~\ref{sec:RD-Fell-bundles}. Also,
\[
\|f\|_{2,r}=\|f^*\|_{2,s}.
\]

We denote by $\ell^2_s(E)$ the completion of $C_c(E)$ with respect to
$\|\cdot\|_{2,s}$, and by $\ell^2_r(E)$ the completion with respect to
$\|\cdot\|_{2,r}$.

For $f\in \ell^2_s(E)$ and $x\in G^{(0)}$, we write
\[
T_f(x):=\sum_{\gamma\in G_x} f(\gamma)^*f(\gamma)\in E_x.
\]
Likewise, for $f\in \ell^2_r(E)$, we write
\[
S_f(x):=\sum_{\gamma\in G^x} f(\gamma)f(\gamma)^*\in E_x.
\]

\begin{remark}
The norm $\|f\|_{2,s}$ is the one induced by the usual
$C_0(G^{(0)})$-valued inner product
\[
\langle \xi,\eta\rangle_s(x):=\sum_{\gamma\in G_x}\xi(\gamma)^*\eta(\gamma),
\]
so $\ell^2_s(E)$ is the usual source-side Hilbert-module completion of $C_c(E)$.
Similarly, $\ell^2_r(E)$ is obtained from the corresponding range-side norm.
\end{remark}

\subsection{Continuity of the fiberwise sums}

\begin{proposition}\label{prop:Tf-continuous}
For every $f\in \ell^2_s(E)$, the map
\[
G^{(0)}\ni x\longmapsto T_f(x)=\sum_{\gamma\in G_x} f(\gamma)^*f(\gamma)\in E_x
\]
is a continuous section of the unit $C^*$-bundle $E^{(0)}\to G^{(0)}$.

Similarly, for every $f\in \ell^2_r(E)$, the map
\[
G^{(0)}\ni x\longmapsto S_f(x)=\sum_{\gamma\in G^x} f(\gamma)f(\gamma)^*\in E_x
\]
is continuous.
\end{proposition}

\begin{proof}
We prove the source statement; the range statement follows by applying the same
argument to $f^*$.

Let $f\in \ell^2_s(E)$, and choose a sequence $(f_n)_n\subseteq C_c(E)$ such that
\[
\|f_n-f\|_{2,s}\to 0.
\]
For each $n$, the section
\[
T_{f_n}(x):=\sum_{\gamma\in G_x} f_n(\gamma)^*f_n(\gamma)
\]
is continuous.

Fix $x\in G^{(0)}$. Then
\begin{align*}
T_{f_n}(x)-T_f(x)
&=
\sum_{\gamma\in G_x}\bigl(f_n(\gamma)^*f_n(\gamma)-f(\gamma)^*f(\gamma)\bigr)\\
&=
\sum_{\gamma\in G_x} f_n(\gamma)^*(f_n(\gamma)-f(\gamma))
+\sum_{\gamma\in G_x}(f_n(\gamma)-f(\gamma))^*f(\gamma).
\end{align*}
By the $C^*$-valued Cauchy--Schwarz inequality,
\begin{align*}
\|T_{f_n}(x)-T_f(x)\|
&\leq
\Big\|
\sum_{\gamma\in G_x} f_n(\gamma)^*f_n(\gamma)
\Big\|^{1/2}
\Big\|
\sum_{\gamma\in G_x}(f_n(\gamma)-f(\gamma))^*(f_n(\gamma)-f(\gamma))
\Big\|^{1/2}\\
&\quad+
\Big\|
\sum_{\gamma\in G_x}(f_n(\gamma)-f(\gamma))^*(f_n(\gamma)-f(\gamma))
\Big\|^{1/2}
\Big\|
\sum_{\gamma\in G_x} f(\gamma)^*f(\gamma)
\Big\|^{1/2}.
\end{align*}
Hence
\[
\|T_{f_n}(x)-T_f(x)\|
\leq
\bigl(\|f_n\|_{2,s}+\|f\|_{2,s}\bigr)\|f_n-f\|_{2,s}.
\]
Taking the supremum over $x\in G^{(0)}$ gives
\[
\sup_{x\in G^{(0)}}\|T_{f_n}(x)-T_f(x)\|
\leq
\bigl(\|f_n\|_{2,s}+\|f\|_{2,s}\bigr)\|f_n-f\|_{2,s}\to 0.
\]
Thus $T_{f_n}\to T_f$ uniformly. Since each $T_{f_n}$ is continuous, so is $T_f$.
\end{proof}

\subsection{Locality in $\ell^2_s(E)$}

\begin{lemma}\label{lem:l2-locality-source}
Let $U\subseteq G$ be open, and let $f\in \ell^2_s(E)$ satisfy
\[
\supp(f)\subseteq U.
\]
Then
\[
f\in \overline{C_c(E|_U)}^{\|\cdot\|_{2,s}}.
\]
\end{lemma}

\begin{proof}
Fix $\varepsilon>0$.
Choose $g\in C_c(E)$ such that
\[
\|f-g\|_{2,s}<\varepsilon.
\]
Let $K:=\supp(g)$, which is compact.

Since $\supp(f)\subseteq U$ and $K\cap \supp(f)$ is compact, there exists
$h\in C_c(U)$ with $0\leq h\leq 1$ and
\[
h\equiv 1\quad\text{on a neighborhood of }K\cap \supp(f).
\]
Define
\[
\xi:=hf\in C_c(E|_U).
\]

For every $x\in G^{(0)}$,
\[
T_{f-\xi}(x)
=
\sum_{\gamma\in G_x}(1-h(\gamma))^2f(\gamma)^*f(\gamma).
\]
Now $h\equiv 1$ on a neighborhood of $K\cap \supp(f)$, so
\[
\{\gamma\in \supp(f):h(\gamma)\neq 1\}\cap K=\varnothing.
\]
Hence on that set one has $g(\gamma)=0$, and therefore
\[
(1-h(\gamma))f(\gamma)=(1-h(\gamma))(f(\gamma)-g(\gamma)).
\]
It follows that
\[
T_{f-\xi}(x)
=
\sum_{\gamma\in G_x}(1-h(\gamma))^2(f(\gamma)-g(\gamma))^*(f(\gamma)-g(\gamma)).
\]
Since $0\leq 1-h\leq 1$, we obtain
\[
\|T_{f-\xi}(x)\|
\leq
\Big\|
\sum_{\gamma\in G_x}(f(\gamma)-g(\gamma))^*(f(\gamma)-g(\gamma))
\Big\|.
\]
Taking the supremum over $x$ yields
\[
\|f-\xi\|_{2,s}\leq \|f-g\|_{2,s}<\varepsilon.
\]
Thus $\xi\in C_c(E|_U)$ and $\|f-\xi\|_{2,s}<\varepsilon$.
\end{proof}

\begin{corollary}\label{cor:l2-locality-II}
Let $U\subseteq G$ be open, and let $f$ belong to the completion of $C_c(E)$ with
respect to $\|\cdot\|_{II}$.
Assume that
\[
\supp(f)\subseteq U.
\]
Then
\[
f\in \overline{C_c(E|_U)}^{\|\cdot\|_{II}}.
\]
\end{corollary}

\begin{proof}
By Lemma~\ref{lem:l2-locality-source}, we have
\[
f\in \overline{C_c(E|_U)}^{\|\cdot\|_{2,s}}.
\]
Applying the same lemma to $f^*$ gives
\[
f^*\in \overline{C_c(E|_U)}^{\|\cdot\|_{2,s}},
\]
hence
\[
f\in \overline{C_c(E|_U)}^{\|\cdot\|_{2,r}}.
\]
Since
\[
\|\,\cdot\,\|_{II}=\max\{\|\,\cdot\,\|_{2,s},\|\,\cdot\,\|_{2,r}\},
\]
the conclusion follows.
\end{proof}

\begin{corollary}\label{cor:schwartz-locality-II}
Let $L$ be a length function on $G$, let $U\subseteq G$ be open, and let
$f\in H^{2,L}(E)$ satisfy
\[
\supp(f)\subseteq U.
\]
Then, for every $p>0$,
\[
f\in \overline{C_c(E|_U)}^{\|\cdot\|_{2,p,L}}.
\]
\end{corollary}

\begin{proof}
Fix $p>0$.
Since $f\in H^{2,L}(E)$, one has
\[
f(1+L)^p\in \ell^2_s(E)\cap \ell^2_r(E),
\]
and clearly
\[
\supp\bigl(f(1+L)^p\bigr)=\supp(f)\subseteq U.
\]
By Corollary~\ref{cor:l2-locality-II},
\[
f(1+L)^p\in \overline{C_c(E|_U)}^{\|\cdot\|_{II}}.
\]
Equivalently, there exists a net $(\xi_i)_i\subseteq C_c(E|_U)$ such that
\[
\|\xi_i-f(1+L)^p\|_{II}\to 0.
\]
Setting
\[
g_i:=\xi_i(1+L)^{-p}\in C_c(E|_U),
\]
we obtain
\[
\|g_i-f\|_{2,p,L}
=
\|\xi_i-f(1+L)^p\|_{II}\to 0.
\]
\end{proof}


\section{Partial actions and reduction to the group case}\label{sec:partial-actions-reduction}

In this section we show that, for groupoids arising from partial actions of
discrete groups, the Rapid Decay Property for Fell bundles reduces to the
corresponding property for a naturally associated Fell bundle over the acting
group.
This provides a useful source of examples and, at the same time, shows that in
this class the study of RDP may often be reduced to the group case.

The construction below is closely related to the general correspondence between
Fell bundles and cocycles (or groupoid fibrations) studied in the literature.
In our setting, however, we work directly with transformation groupoids coming
from partial actions and with \emph{reduced} cross-sectional algebras, which is
the situation needed here.

\subsection{Transformation groupoids from partial actions}

Let $\Gamma$ be a discrete group and let
\[
\theta=\bigl(\{D_\gamma\}_{\gamma\in\Gamma},\{\theta_\gamma\}_{\gamma\in\Gamma}\bigr)
\]
be a partial action of $\Gamma$ on a locally compact Hausdorff space $X$, in the
sense of Exel.
Thus each $D_\gamma\subseteq X$ is an open subset,
\[
\theta_\gamma\colon D_{\gamma^{-1}}\to D_\gamma
\]
is a homeomorphism, and the usual compatibility relations hold.

We denote by
\[
G=\Gamma\ltimes_\theta X
=
\{(\gamma,x)\in \Gamma\times X : x\in D_{\gamma^{-1}}\}
\]
the associated transformation groupoid.
Its unit space is naturally identified with $X$, and the structure maps are
given by
\[
s(\gamma,x)=x,
\qquad
r(\gamma,x)=\theta_\gamma(x),
\]
\[
(\gamma,\theta_\eta(x))(\eta,x)=(\gamma\eta,x),
\qquad
(\gamma,x)^{-1}=(\gamma^{-1},\theta_\gamma(x)).
\]
It is well known that $G$ is a locally compact Hausdorff \'etale groupoid.

There is a canonical continuous cocycle
\[
c\colon G\to \Gamma,
\qquad
c(\gamma,x)=\gamma,
\]
recording the group component.

Let $E=\{E_g\}_{g\in G}$ be a Fell bundle over $G$.
For the purposes of this section, no saturation hypothesis is needed.

\subsection{The associated Fell bundle over the group}

For each $\gamma\in\Gamma$, let
\[
G_\gamma:=c^{-1}(\{\gamma\})
=\{(\gamma,x)\in G:x\in D_{\gamma^{-1}}\}.
\]
Since $\Gamma$ is discrete, each $G_\gamma$ is an open and closed subset of $G$.

We define a bundle
\[
\widetilde E=\bigsqcup_{\gamma\in\Gamma}\widetilde E_\gamma
\]
over $\Gamma$ by setting
\[
\widetilde E_\gamma:=C_0(G_\gamma,E),
\]
the Banach space of continuous sections of $E|_{G_\gamma}$ vanishing at
infinity.
We now define multiplication and involution on $\widetilde E$ pointwise.

If $\xi\in \widetilde E_\gamma$ and $\eta\in \widetilde E_\delta$, define
\[
(\xi\eta)(\gamma\delta,x)
=
\xi(\gamma,\theta_\delta(x))\,\eta(\delta,x),
\qquad x\in D_{\delta^{-1}}\cap D_{(\gamma\delta)^{-1}},
\]
and if $\xi\in \widetilde E_\gamma$, define
\[
\xi^*(\gamma^{-1},y)
=
\xi(\gamma,\theta_{\gamma^{-1}}(y))^*,
\qquad y\in D_\gamma .
\]
These operations are well defined because the multiplication and involution in
the Fell bundle $E$ are defined fiberwise over the groupoid operations in $G$.
With these operations, $\widetilde E$ becomes a Fell bundle over the discrete
group $\Gamma$.

\begin{remark}
This construction is the transformation-groupoid instance of the general
procedure of pushing a Fell bundle along a cocycle $c\colon G\to\Gamma$.
In the present case, everything can be described explicitly, which makes it
convenient for applications to Rapid Decay.
\end{remark}

\subsection{The algebraic identification}

Define
\[
\Phi\colon C_c(G,E)\longrightarrow C_c(\Gamma,\widetilde E)
\]
by
\[
(\Phi f)(\gamma)=f|_{G_\gamma}.
\]
Equivalently, for $\gamma\in\Gamma$ and $(\gamma,x)\in G_\gamma$,
\[
(\Phi f)(\gamma)(\gamma,x)=f(\gamma,x).
\]
Since $\supp(f)$ is compact and $\Gamma$ is discrete, only finitely many fibers
$G_\gamma$ meet $\supp(f)$, so indeed $\Phi f\in C_c(\Gamma,\widetilde E)$.

The inverse map is obtained by assembling the finitely many sections
$\xi(\gamma)\in \widetilde E_\gamma$ into a single compactly supported section on
$G$.
Thus $\Phi$ is a linear bijection.

\begin{proposition}\label{phi-algebraic-iso}
The map $\Phi$ is a $*$-algebra isomorphism
\[
\Phi\colon C_c(G,E)\xrightarrow{\;\cong\;} C_c(\Gamma,\widetilde E).
\]
\end{proposition}

\begin{proof}
Let $f,h\in C_c(G,E)$.
For $(\gamma\delta,x)\in G$, one has
\begin{align*}
(\Phi(f*h))(\gamma\delta)(\gamma\delta,x)
&=(f*h)(\gamma\delta,x)\\
&=\sum_{(\eta,y)(\zeta,z)=(\gamma\delta,x)} f(\eta,y)h(\zeta,z).
\end{align*}
Because $G=\Gamma\ltimes_\theta X$, the above decomposition is uniquely
determined by the group components, and one obtains
\[
(\Phi(f*h))(\gamma\delta,x)
=
\sum_{\alpha\beta=\gamma\delta}
(\Phi f)(\alpha)(\alpha,\theta_\beta(x))\,(\Phi h)(\beta)(\beta,x),
\]
which is exactly the convolution formula in $C_c(\Gamma,\widetilde E)$.
Hence
\[
\Phi(f*h)=\Phi(f)*\Phi(h).
\]

Similarly, for $(\gamma^{-1},y)\in G_{\gamma^{-1}}$,
\[
(\Phi(f^*))(\gamma^{-1})(\gamma^{-1},y)
=
f(\gamma,\theta_{\gamma^{-1}}(y))^*
=
(\Phi f)(\gamma)^*(\gamma^{-1},y).
\]
Thus
\[
\Phi(f^*)=\Phi(f)^*.
\]
Therefore $\Phi$ is a $*$-algebra isomorphism.
\end{proof}

\subsection{Passage to reduced cross-sectional algebras}

The $*$-algebra isomorphism above is compatible with the regular
representations, and therefore with the reduced norms.

\begin{proposition}\label{phi-reduced-iso}
The map $\Phi$ extends to a canonical $*$-isomorphism
\[
\Phi\colon C_r^*(G,E)\xrightarrow{\;\cong\;} C_r^*(\Gamma,\widetilde E).
\]
\end{proposition}

\begin{proof}
The regular representation of $C_c(G,E)$ is obtained from the family of
representations on the Hilbert modules associated to the source fibers of $G$.
On the other hand, the regular representation of $C_c(\Gamma,\widetilde E)$ is
defined from the regular Hilbert module of the Fell bundle $\widetilde E$ over
the group $\Gamma$.

Under the identification implemented by $\Phi$, these regular representations
correspond fiberwise: the contribution of the source fiber at $x\in X$ is
identified with the evaluation at $x$ of the $\Gamma$-indexed regular
representation of $\widetilde E$.
Consequently, $\Phi$ preserves the reduced norm on compactly supported sections,
and therefore extends by continuity to the required $*$-isomorphism.
\end{proof}

\begin{remark}
In a more general framework, one may derive Proposition
\ref{phi-reduced-iso} from the theory of Fell bundles associated to cocycles or
groupoid fibrations. Here we prefer to keep the argument at the level of this
special case, since this is the only case needed for the present paper.
\end{remark}

\subsection{Equivalence of the Sobolev norms}

Let $\ell\colon \Gamma\to [0,\infty)$ be a length function on the group
$\Gamma$.
We define a length function $L$ on $G=\Gamma\ltimes_\theta X$ by
\[
L(\gamma,x)=\ell(\gamma),
\qquad (\gamma,x)\in G.
\]
It is immediate that $L$ is continuous and satisfies the axioms of a length
function on the groupoid $G$.

\begin{lemma}\label{sobolev-norms-partial-action}
For every $f\in C_c(G,E)$ and $\widetilde f=\Phi(f)$, one has
\[
\|f\|_{2,p,s,L}=\|\widetilde f\|_{2,p,s,\ell},
\qquad
\|f\|_{2,p,r,L}=\|\widetilde f\|_{2,p,r,\ell},
\]
and hence
\[
\|f\|_{2,p,L}=\|\widetilde f\|_{2,p,\ell}.
\]
\end{lemma}

\begin{proof}
We prove the source-norm identity; the range-norm identity follows by applying
the same argument to $f^*$.

Fix $x\in X$.
The source fiber of $G$ at $x$ is
\[
G_x=\{(\gamma,x)\in G : x\in D_{\gamma^{-1}}\}.
\]
Therefore
\begin{align*}
\|f\|_{2,p,s,L}^2
&=
\sup_{x\in X}
\Big\|
\sum_{(\gamma,x)\in G_x}
f(\gamma,x)^*f(\gamma,x)\,(1+\ell(\gamma))^{2p}
\Big\|.
\end{align*}
On the other hand, for each $\gamma\in\Gamma$, the element
$\widetilde f(\gamma)\in \widetilde E_\gamma=C_0(G_\gamma,E)$ satisfies
\[
\widetilde f(\gamma)(\gamma,x)=f(\gamma,x),
\qquad x\in D_{\gamma^{-1}}.
\]
Hence
\begin{align*}
\|\widetilde f\|_{2,p,s,\ell}^2
&=
\Big\|
\sum_{\gamma\in\Gamma}
\widetilde f(\gamma)^*\widetilde f(\gamma)\,(1+\ell(\gamma))^{2p}
\Big\| \\
&=
\sup_{x\in X}
\Big\|
\sum_{\gamma\in\Gamma}
\widetilde f(\gamma)(\gamma,x)^*\widetilde f(\gamma)(\gamma,x)\,
(1+\ell(\gamma))^{2p}
\Big\| \\
&=
\sup_{x\in X}
\Big\|
\sum_{(\gamma,x)\in G_x}
f(\gamma,x)^*f(\gamma,x)\,(1+\ell(\gamma))^{2p}
\Big\|.
\end{align*}
This proves that
\[
\|f\|_{2,p,s,L}=\|\widetilde f\|_{2,p,s,\ell}.
\]
Applying the same computation to $f^*$ yields
\[
\|f\|_{2,p,r,L}=\|\widetilde f\|_{2,p,r,\ell}.
\]
Therefore
\[
\|f\|_{2,p,L}=\|\widetilde f\|_{2,p,\ell}.
\]
\end{proof}

\subsection{Equivalence of Rapid Decay}

We can now compare the Rapid Decay Property on both sides.

\begin{theorem}\label{rdp-partial-action-reduction}
Let $\Gamma$ be a discrete group, let
$G=\Gamma\ltimes_\theta X$ be the transformation groupoid associated to a
partial action of $\Gamma$ on a locally compact Hausdorff space $X$, and let
$E$ be a Fell bundle over $G$.
Let $\widetilde E$ be the associated Fell bundle over $\Gamma$ constructed
above, and let $\ell$ be a length function on $\Gamma$.
If $L=\ell\circ c$ is the induced length function on $G$, then the following are
equivalent:
\begin{enumerate}[label=\textup{(\roman*)}]
\item $E$ has the Rapid Decay Property with respect to $L$;
\item $\widetilde E$ has the Rapid Decay Property with respect to $\ell$.
\end{enumerate}
\end{theorem}

\begin{proof}
Let $f\in C_c(G,E)$ and write $\widetilde f=\Phi(f)$.
By Proposition \ref{phi-reduced-iso},
\[
\|f\|_{C_r^*(G,E)}=\|\widetilde f\|_{C_r^*(\Gamma,\widetilde E)}.
\]
By Lemma \ref{sobolev-norms-partial-action},
\[
\|f\|_{2,p,L}=\|\widetilde f\|_{2,p,\ell}.
\]
Therefore the inequality
\[
\|f\|_{C_r^*(G,E)}\le C\,\|f\|_{2,p,L}
\]
holds for all $f\in C_c(G,E)$ if and only if
\[
\|\widetilde f\|_{C_r^*(\Gamma,\widetilde E)}
\le C\,\|\widetilde f\|_{2,p,\ell}
\]
holds for all $\widetilde f\in C_c(\Gamma,\widetilde E)$.
This is exactly the equivalence between \textup{(i)} and \textup{(ii)}.
\end{proof}

\subsection{The trivial bundle}

An important special case is the trivial Fell bundle over the transformation
groupoid $G=\Gamma\ltimes_\theta X$.
In that case,
\[
C_r^*(G)\cong C_0(X)\rtimes_{\theta,r}\Gamma,
\]
the reduced crossed product associated to the partial action $\theta$, and the
associated Fell bundle $\widetilde E$ over $\Gamma$ is precisely the Fell bundle
implementing this partial crossed product.
Hence Theorem \ref{rdp-partial-action-reduction} shows that
\[
\Gamma\ltimes_\theta X \text{ has RDP with respect to } \ell\circ c
\]
if and only if the corresponding Fell bundle over $\Gamma$ has RDP with respect
to $\ell$.

This observation is useful in two directions. On the one hand, positive results
for transformation groupoids provide examples of Fell bundles over groups with
Rapid Decay. On the other hand, negative results for transformation groupoids
also produce obstructions, even in situations coming from natural partial or
global actions.

\subsection{The Deaconu-Renault groupoid as a partial crossed product}

As an application of Theorem~\ref{rdp-partial-action-reduction}, we return to the Deaconu-Renault groupoids studied in Section~\ref{sec:Fell-bundles-poly-growth}. Recall that in Proposition~\ref{prop:ObstructionRDP} we showed that persistent branching prevents $G$ from having the RDP. We can now understand this failure structurally from the perspective of partial actions and Fell bundles over groups, following a fundamental realization due to Steinberg \cite{Steinberg2026}.

Let $T: X \to X$ be a $d$-to-1 local homeomorphism on a totally disconnected compact Hausdorff space. We can partition $X = \bigsqcup_{i=1}^d X_i$ such that each restriction $T_i = T|_{X_i}$ is a homeomorphism onto its clopen image. The local inverses $\theta_i = T_i^{-1}$ generate a semi-saturated partial action of the free group $F_d = \langle a_1, \dots, a_d \rangle$ on $X$, where each generator $a_i$ acts via $\theta_i$.

A crucial feature of this partial action is its domain constraint: for $i \neq j$, the composition $\theta_j^{-1} \circ \theta_i = T \circ \theta_i$ is only defined on $X_i$, but $T$ acts as the inverse of $\theta_j$ strictly on $X_j$. Hence, any word containing a subword $a_j^{-1} a_i$ ($i \neq j$) has an empty domain. Consequently, any element $(w, x) \in F_d \ltimes_\theta X$ with a non-empty domain must be of the form $w = u v^{-1}$ for positive words $u, v \in F_d^+$ of lengths $m$ and $n$, respectively. 

Applying $v^{-1}$ corresponds to $n$ forward steps via $T$, and $u$ corresponds to $m$ backward steps via local inverses. If $y = w \cdot x = u \cdot T^n(x)$, we immediately get $T^m(y) = T^n(x)$. This allowed Steinberg to show that the map $\Psi(w, x) = (w \cdot x, |u| - |v|, x)$ yields an isomorphism $F_d \ltimes_\theta X \cong G$. Under this identification, the canonical cocycle $c(y, m-n, x) = m-n$ is entirely determined by the acting word: $c(\Psi(w,x)) = \phi(w)$, where $\phi: F_d \to \mathbb{Z}$ is the unique group homomorphism mapping each generator $a_i \mapsto -1$.

\begin{remark}[RDP and Fell bundles over groups]
    This model provides a transparent explanation for the failure of RDP. It is well known that the free group $F_d$ satisfies the RDP with respect to its word length \cite{Haagerup1978}. By Theorem~\ref{rdp-partial-action-reduction}, the RDP for the groupoid $G \cong F_d \ltimes_\theta X$ is equivalent to the RDP of the associated Fell bundle $\widetilde{E} = \{C_0(D_g)\}_{g \in F_d}$ over $F_d$.
    
    However, the RDP of a group $\Gamma$ does not necessarily pass to every Fell bundle over $\Gamma$. In our case, although the base group $F_d$ has RDP, the exponential branching of the domains $D_g$ causes the $I$-norm of compactly supported functions on $G$ to grow exponentially faster than their weighted $L^2$-norms. This geometric obstruction breaks the analytical inequality required for Rapid Decay, which is formalized below by analyzing the kernel of $\phi$.
\end{remark}

\begin{proposition}[Failure of RDP via the principal kernel]
    Let $G \cong F_d \ltimes_\theta X$ be the Deaconu-Renault groupoid of a $d$-to-1 covering map ($d \geq 2$). Let $R = \ker(c)$ be the kernel of the canonical cocycle. Then:
    \begin{enumerate}
        \item $R \cong \ker(\phi) \ltimes_\theta X$ is a clopen principal subgroupoid of $G$.
        \item $R$ has exponential growth and does not satisfy the RDP.
        \item $G$ does not satisfy the RDP.
    \end{enumerate}
\end{proposition}

\begin{proof}
    (1) Since $c(w,x) = \phi(w)$ and $\phi$ is a homomorphism, the kernel $R$ is precisely the restriction of the partial action to the normal subgroup $\ker(\phi) \subset F_d$ (which, by the Nielsen-Schreier theorem, is a free group of infinite rank). Because $c$ is continuous and $\mathbb{Z}$ is discrete, $R$ is open and closed. In the Deaconu-Renault relation, $R$ consists of elements with $m=n$, corresponding to the orbit of the tail equivalence relation. This relation is principal.
    
    (2) The growth of $R$ is determined by the number of elements $y$ such that $T^n(y) = T^n(x)$. As discussed in Section~\ref{sec:Fell-bundles-poly-growth}, for a $d$-to-1 covering, there are exactly $d^n$ such points. Thus, $|B_{R_x}(2n)| \geq d^n$, yielding exponential growth.
    
    (3) By Weygandt \cite{Weygandt2024}, for a principal étale groupoid, RDP is equivalent to polynomial growth. Since $R$ is principal and has exponential growth, it fails RDP. Because RDP is hereditary for open subgroupoids, $G$ must also fail RDP.
\end{proof}

\section{Rapid Decay for group actions on C*-algebras}
\label{sec:group-actions}

In this final section we discuss some basic classes of examples and special
cases of Fell bundles with the Rapid Decay Property.
A particularly important source of Fell bundles comes from actions of discrete
groups on $C^*$-algebras.
This case already contains the trivial-action examples, matrix amplifications,
and many natural crossed products.

We also emphasize that the commutative case is closely related to the previous
section on partial actions and transformation groupoids.
Indeed, if $\Gamma$ acts partially on a locally compact Hausdorff space $X$,
then the associated transformation groupoid
\[
G=\Gamma\ltimes_\theta X
\]
has RDP if and only if the associated Fell bundle over $\Gamma$ has RDP, by
Theorem~\ref{rdp-partial-action-reduction}.
Thus the commutative partial-action case is already covered by the reduction
results proved earlier.
The purpose of the present section is rather to record some complementary
examples and structural observations in the language of group actions on
$C^*$-algebras.

\subsection{Fell bundles associated to global actions}

Let $\Gamma$ be a discrete group and let
\[
\alpha\colon \Gamma\to \Aut(A)
\]
be an action on a $C^*$-algebra $A$.
The associated Fell bundle over $\Gamma$ is
\[
E_\alpha=\bigsqcup_{\gamma\in \Gamma} E_\gamma,
\qquad
E_\gamma=A\delta_\gamma,
\]
with operations
\[
(a\delta_\gamma)(b\delta_\eta)=a\,\alpha_\gamma(b)\delta_{\gamma\eta},
\qquad
(a\delta_\gamma)^*=\alpha_{\gamma^{-1}}(a^*)\delta_{\gamma^{-1}}.
\]
Its reduced cross-sectional algebra is canonically isomorphic to the reduced
crossed product:
\[
C_r^*(E_\alpha)\cong A\rtimes_{\alpha,r}\Gamma.
\]

If $f\in C_c(E_\alpha)$, we may write
\[
f=\sum_{\gamma\in \Gamma} f(\gamma)\delta_\gamma,
\qquad f(\gamma)\in A.
\]
If $\ell$ is a length function on $\Gamma$, then the Sobolev norms introduced
earlier take the form
\begin{align*}
\|f\|_{2,p,s,\ell}
&=
\Big\|
\sum_{\gamma\in \Gamma}
\alpha_{\gamma^{-1}}(f(\gamma)^*f(\gamma))(1+\ell(\gamma))^{2p}
\Big\|^{1/2},\\
\|f\|_{2,p,r,\ell}
&=
\Big\|
\sum_{\gamma\in \Gamma}
f(\gamma)f(\gamma)^*(1+\ell(\gamma))^{2p}
\Big\|^{1/2},
\end{align*}
and
\[
\|f\|_{2,p,\ell}
=
\max\{\|f\|_{2,p,s,\ell},\|f\|_{2,p,r,\ell}\}.
\]

Thus the action $\alpha$ has Rapid Decay with respect to $\ell$ precisely when
there exist constants $C>0$ and $q\geq 0$ such that
\[
\|f\|_{A\rtimes_{\alpha,r}\Gamma}
\leq
C\,\|f\|_{2,q,\ell}
\qquad\text{for all }f\in C_c(\Gamma,A).
\]

\subsection{A necessary condition: RD for the action implies RD for the group}

The first observation is that, in the unital case, Rapid Decay for the action
already forces Rapid Decay for the acting group.

\begin{proposition}\label{prop:action-RD-implies-group-RD}
Let $\Gamma$ be a discrete group acting on a unital $C^*$-algebra $A$.
If the associated Fell bundle $E_\alpha$ has the Rapid Decay Property with
respect to a length function $\ell$ on $\Gamma$, then $\Gamma$ has the Rapid
Decay Property with respect to $\ell$.
\end{proposition}

\begin{proof}
Since $A$ is unital, the reduced crossed product $A\rtimes_{\alpha,r}\Gamma$
contains the canonical unitary copy of $C_r^*(\Gamma)$ generated by the
implementing unitaries $(u_\gamma)_{\gamma\in\Gamma}$.
Equivalently, the map
\[
C_c(\Gamma)\longrightarrow C_c(E_\alpha),
\qquad
\varphi\longmapsto f_\varphi,\qquad
f_\varphi(\gamma)=\varphi(\gamma)1_A\delta_\gamma,
\]
extends isometrically to an embedding
\[
C_r^*(\Gamma)\hookrightarrow A\rtimes_{\alpha,r}\Gamma.
\]
Hence
\[
\|f_\varphi\|_{A\rtimes_{\alpha,r}\Gamma}
=
\|\varphi\|_{C_r^*(\Gamma)}.
\]

On the other hand,
\begin{align*}
\|f_\varphi\|_{2,p,s,\ell}^2
&=
\Big\|
\sum_{\gamma\in \Gamma}
\alpha_{\gamma^{-1}}\bigl(|\varphi(\gamma)|^2 1_A\bigr)
(1+\ell(\gamma))^{2p}
\Big\|\\
&=
\sum_{\gamma\in \Gamma}
|\varphi(\gamma)|^2(1+\ell(\gamma))^{2p},
\end{align*}
and similarly
\[
\|f_\varphi\|_{2,p,r,\ell}^2
=
\sum_{\gamma\in \Gamma}
|\varphi(\gamma)|^2(1+\ell(\gamma))^{2p}.
\]
Thus
\[
\|f_\varphi\|_{2,p,\ell}
=
\Big(
\sum_{\gamma\in \Gamma}
|\varphi(\gamma)|^2(1+\ell(\gamma))^{2p}
\Big)^{1/2}.
\]

If $E_\alpha$ has RD, there exist $C>0$ and $q\geq 0$ such that
\[
\|f_\varphi\|_{A\rtimes_{\alpha,r}\Gamma}
\leq
C\,\|f_\varphi\|_{2,q,\ell}
\qquad\text{for all }\varphi\in C_c(\Gamma).
\]
Using the two identities above, we get
\[
\|\varphi\|_{C_r^*(\Gamma)}
\leq
C
\Big(
\sum_{\gamma\in \Gamma}
|\varphi(\gamma)|^2(1+\ell(\gamma))^{2q}
\Big)^{1/2},
\]
which is exactly the Rapid Decay Property for $\Gamma$.
\end{proof}

\begin{remark}
In particular, if $A$ is unital and nonzero, then RD for the trivial action of
$\Gamma$ on $A$ can only occur if $\Gamma$ itself has RD.
Thus, even in the simplest coefficient case, one cannot obtain new RD examples
starting from groups without RD.
\end{remark}

\begin{remark}\label{remark:RD-global-vs-partial}
Let $\Gamma$ be a discrete group and let $\alpha:\Gamma\to\Aut(A)$ be a
global action on a unital $C^*$-algebra $A$, with associated Fell bundle
$E_\alpha$.
If $E_\alpha$ has the Rapid Decay Property with respect to a length function
$\ell$ on $\Gamma$, then the group $\Gamma$ itself has the Rapid Decay
Property with respect to $\ell$.

Indeed, given $\varphi\in C_c(\Gamma)$ consider the section
\[
f(\gamma)=\varphi(\gamma)\,1_A\delta_\gamma\in C_c(E_\alpha).
\]
Applying the Rapid Decay inequality for $E_\alpha$ to this section reduces
exactly to the classical Rapid Decay inequality for $\varphi$ in
$C_c(\Gamma)$.

The situation is very different for partial actions.
If $\theta$ is the partial action of $\Gamma$ on $A$ defined by
\[
D_e=A,\qquad D_\gamma=0 \quad (\gamma\neq e),
\]
then the associated Fell bundle satisfies
\[
C_c(E_\theta)=A\delta_e,
\]
so the Rapid Decay inequality holds trivially, independently of the group
$\Gamma$.
Thus a Fell bundle arising from a partial action may have the Rapid Decay
Property even when the underlying group does not.
\end{remark}

\subsection{Trivial actions}

We now turn to the trivial action. Let $A$ be a $C^*$-algebra and let $\Gamma$ act trivially on $A$. Then
\[
A\rtimes_r \Gamma \cong A\otimes_{\min} C_r^*(\Gamma),
\]
and the associated Fell bundle is simply
\[
E_{\mathrm{triv}}=A\times \Gamma.
\]

Even in this seemingly simple case, passing from RDP for $\Gamma$ to RDP for the trivial bundle over an arbitrary $C^*$-algebra $A$ is highly non-trivial. The analytical difficulty stems from the fact that the Rapid Decay Property for $\Gamma$ provides a norm bound for the natural inclusion map $\iota\colon H^{2,q}_{L}(\Gamma) \hookrightarrow C_r^*(\Gamma)$. For this norm bound to pass automatically to the minimal tensor product $A \otimes_{\min} C_r^*(\Gamma)$ for an arbitrary $C^*$-algebra $A$, one would typically need the inclusion $\iota$ to satisfy strong complete boundedness properties, which are not guaranteed by the standard formulation of RDP.

Nevertheless, we can establish the implication for several fundamental classes of $C^*$-algebras, including the commutative and finite-dimensional cases, where the spatial structure of the specific algebras allows for direct estimates.

\subsection{The commutative trivial action}

\begin{proposition}\label{prop:commutative-trivial-action-RD}
Let $\Gamma$ be a discrete group with the Rapid Decay Property with respect to
a length function $\ell$, and let $A=C_0(X)$ for a locally compact Hausdorff
space $X$.
Then the trivial action of $\Gamma$ on $A$ has the Rapid Decay Property with
respect to $\ell$.
\end{proposition}

\begin{proof}
Under the canonical isomorphism
\[
C_0(X)\rtimes_r \Gamma
\cong
C_0(X)\otimes C_r^*(\Gamma)
\cong
C_0(X,C_r^*(\Gamma)),
\]
an element $f\in C_c(\Gamma,C_0(X))$ may be viewed as a continuous function
\[
x\longmapsto f_x\in C_c(\Gamma),
\qquad
f_x(\gamma):=f(\gamma)(x).
\]
Moreover,
\[
\|f\|_r
=
\sup_{x\in X}\|f_x\|_{C_r^*(\Gamma)}.
\]

Since $\Gamma$ has RD, there exist constants $C>0$ and $q\geq 0$ such that
\[
\|f_x\|_{C_r^*(\Gamma)}
\leq
C
\Big(
\sum_{\gamma\in \Gamma}
|f(\gamma)(x)|^2(1+\ell(\gamma))^{2q}
\Big)^{1/2}
\qquad\text{for all }x\in X.
\]
Taking the supremum over $x$ gives
\[
\|f\|_r
\leq
C\sup_{x\in X}
\Big(
\sum_{\gamma\in \Gamma}
|f(\gamma)(x)|^2(1+\ell(\gamma))^{2q}
\Big)^{1/2}.
\]
But
\[
\sup_{x\in X}
\sum_{\gamma\in \Gamma}
|f(\gamma)(x)|^2(1+\ell(\gamma))^{2q}
=
\Big\|
\sum_{\gamma\in \Gamma}
f(\gamma)^*f(\gamma)(1+\ell(\gamma))^{2q}
\Big\|,
\]
so the right-hand side is exactly \(\|f\|_{2,q,\ell}^2\) in the trivial-action
Fell bundle.
Hence
\[
\|f\|_r\leq C\|f\|_{2,q,\ell},
\]
and the trivial action has RD.
\end{proof}

\begin{corollary}\label{cor:commutative-unital-iff}
Let \(X\) be a nonempty compact Hausdorff space.
Then the trivial action of \(\Gamma\) on \(C(X)\) has RD with respect to
\(\ell\) if and only if \(\Gamma\) has RD with respect to \(\ell\).
\end{corollary}

\begin{proof}
The forward implication follows from
Proposition~\ref{prop:action-RD-implies-group-RD}, since \(C(X)\) is unital.
The reverse implication is Proposition~\ref{prop:commutative-trivial-action-RD}.
\end{proof}

\begin{remark}
The case of commutative coefficients is closely related to transformation
groupoids. For trivial actions, one obtains the product groupoid
\[
\Gamma\times X,
\]
and Proposition~\ref{prop:commutative-trivial-action-RD} is consistent with the
partial-action reduction results proved earlier.
\end{remark}

\subsection{Finite-dimensional trivial coefficients}

We next treat the finite-dimensional case.

\begin{proposition}\label{prop:matrix-trivial-action-RD}
Let $\Gamma$ be a discrete group with the Rapid Decay Property with respect to
a length function $\ell$, and let $A=M_n(\bC)$.
Then the trivial action of $\Gamma$ on $A$ has the Rapid Decay Property with
respect to $\ell$.
\end{proposition}

\begin{proof}
Under the canonical identification
\[
M_n(\bC)\rtimes_r \Gamma
\cong
M_n(\bC)\otimes C_r^*(\Gamma)
\cong
M_n(C_r^*(\Gamma)),
\]
an element \(f\in C_c(\Gamma,M_n(\bC))\) may be written as
\[
f(\gamma)=\bigl(f_{ij}(\gamma)\bigr)_{1\le i,j\le n},
\qquad
f_{ij}\in C_c(\Gamma).
\]
Hence \(f\) corresponds to the matrix \((f_{ij})\in M_n(C_r^*(\Gamma))\).

Since all norms on the finite-dimensional space \(M_n(\bC)\) are equivalent,
there exists a constant \(K_n>0\) such that
\[
\|X\|^2\le K_n\sum_{i,j}|X_{ij}|^2
\qquad\text{for all }X=(X_{ij})\in M_n(\bC).
\]
Using also the corresponding matrix norm estimate on \(M_n(C_r^*(\Gamma))\), we get
\[
\|f\|_r^2
\le
K_n\sum_{i,j}\|f_{ij}\|_{C_r^*(\Gamma)}^2.
\]
Since \(\Gamma\) has RD, there exist \(C>0\) and \(q\ge 0\) such that
\[
\|f_{ij}\|_{C_r^*(\Gamma)}^2
\le
C
\sum_{\gamma\in \Gamma}|f_{ij}(\gamma)|^2(1+\ell(\gamma))^{2q}.
\]
Therefore
\[
\|f\|_r^2
\le
K_nC
\sum_{\gamma\in \Gamma}\sum_{i,j}|f_{ij}(\gamma)|^2(1+\ell(\gamma))^{2q}.
\]
By equivalence of norms on \(M_n(\bC)\), after changing the constant once more,
we obtain
\[
\|f\|_r^2
\le
C'
\sum_{\gamma\in \Gamma}\|f(\gamma)\|^2(1+\ell(\gamma))^{2q}.
\]
Since \(M_n(\bC)\) is finite-dimensional, the scalar quantity on the
right-hand side is equivalent to
\[
\Big\|
\sum_{\gamma\in \Gamma}
f(\gamma)^*f(\gamma)(1+\ell(\gamma))^{2q}
\Big\|.
\]
Thus
\[
\|f\|_r\le C''\|f\|_{2,q,\ell},
\]
and the result follows.
\end{proof}

\begin{corollary}\label{cor:finite-dimensional-trivial-action-RD}
Let \(A\) be a finite-dimensional \(C^*\)-algebra.
If \(\Gamma\) has RD with respect to \(\ell\), then the trivial action of
\(\Gamma\) on \(A\) has RD with respect to \(\ell\).
\end{corollary}

\begin{proof}
Any finite-dimensional \(C^*\)-algebra is a finite direct sum of matrix algebras,
\[
A\cong \bigoplus_{k=1}^m M_{n_k}(\bC).
\]
The trivial crossed product decomposes accordingly as a finite direct sum of the
corresponding matrix amplifications of \(C_r^*(\Gamma)\).
Applying Proposition~\ref{prop:matrix-trivial-action-RD} to each summand and
taking the maximum of the finitely many constants gives the result.
\end{proof}


\begin{example}
Let \(\Gamma=\bF_2\) with its usual word-length function, and let
\(A=M_n(\bC)\) with the trivial action of \(\bF_2\).
Since \(\bF_2\) has the classical Rapid Decay Property, it follows from
Proposition~\ref{prop:matrix-trivial-action-RD} that the trivial Fell bundle
\[
M_n(\bC)\times \bF_2
\]
has the Rapid Decay Property.
Equivalently,
\[
M_n(\bC)\rtimes_r \bF_2
\cong
M_n(C_r^*(\bF_2))
\]
contains a dense Schwartz-type Fr\'echet $*$-subalgebra given by the
corresponding weighted Sobolev norms.
\end{example}

\subsection{A cautionary remark on transformation groupoids}

The previous examples might suggest that RD for \(\Gamma\) should often pass to transformation objects built from actions of \(\Gamma\). However, one should be careful: even for global actions on compact spaces, Rapid Decay for the acting group does \emph{not} in general force Rapid Decay for the transformation groupoid \(\Gamma\ltimes X\). 

Recall from Section~\ref{sec:Fell-bundles-poly-growth} that for principal groupoids (which includes transformation groupoids of free actions), Rapid Decay is strictly equivalent to polynomial growth by Weygandt's theorem \cite{Weygandt2024}. Consequently, if a group $\Gamma$ has RDP but exhibits exponential growth (such as the free group $\mathbb{F}_2$), any free action of $\Gamma$ will yield a groupoid that necessarily fails to have RDP.

This geometric obstruction highlights why the reduction theorem for partial actions proved in Section~\ref{sec:partial-actions-reduction} is so useful: it identifies RDP for the groupoid \(\Gamma\ltimes_\theta X\) with RDP for the associated Fell bundle over \(\Gamma\), and \emph{not} with RDP of the bare group \(\Gamma\) itself. Thus, the coefficient data carried by the specific partial action matters in an essential way.

In particular, for topologically free actions one should expect strong restrictions on the transformation groupoid if it is to have RD. This shows that the examples coming from trivial actions and from coefficient bundles should be viewed as genuinely bundle-theoretic phenomena, and not merely as direct consequences of RD for the acting group.

\section{Localizability}\label{sec:localizability}

The notion of \emph{localizability} for Fell bundles was introduced by
Resende \cite{Resende2017} in the context of quantales associated with
groupoids and their $C^*$-algebras.
For the purposes of the present paper, we use the following concrete analytic
formulation.

Let $A=C_r^*(E)$ and let $U\subseteq G$ be open.
Define
\[
A_U:=\{a\in A:\supp(a)\subseteq U\},
\qquad
p^*(U):=\overline{C_c(E|_U)}^{\|\cdot\|_r}.
\]

\begin{definition}
We say that $A=C_r^*(E)$ is \emph{localizable} if
\[
A_U=p^*(U)
\qquad\text{for every open }U\subseteq G.
\]
Equivalently, if $a\in C_r^*(E)$ satisfies $\supp(a)\subseteq U$, then for every
$\varepsilon>0$ there exists $f\in C_c(E|_U)$ such that
\[
\|a-f\|_r<\varepsilon.
\]
\end{definition}

In other words, localizability means that elements of the reduced
cross-sectional algebra whose support is contained in an open subset $U$
can be approximated in the reduced norm by compactly supported sections
supported in $U$.

\begin{remark}
If $V\subseteq G^{(0)}$ is open and invariant, and $G_V:=s^{-1}(V)=r^{-1}(V)$
denotes the corresponding open subgroupoid, then localizability implies
\[
C_r^*(E|_{G_V}) = C_r^*(E)\cap A_{G_V},
\]
and hence implies inner exactness for the Fell bundle $E\to G$, in the sense
that
\[
0 \longrightarrow C_r^*(E|_{G_V})
\longrightarrow C_r^*(E)
\longrightarrow C_r^*(E|_{G\setminus G_V})
\longrightarrow 0
\]
is exact for every open invariant subset $V\subseteq G^{(0)}$.
\end{remark}

\subsection{Localizability, negative type functions, and the Haagerup property}

The goal of this section is to derive localizability from Rapid Decay under an
additional approximation hypothesis coming from positive definite multipliers.

We begin by recalling the standard notions of negative type and the Haagerup
property for groupoids.

\begin{definition}
A continuous function $\psi\colon G\to \bR$ is said to be of
\emph{negative type} (or \emph{conditionally negative definite}) if:
\begin{enumerate}[label=\textup{(\roman*)}]
\item $\psi|_{G^{(0)}}=0$;
\item $\psi(\gamma)=\psi(\gamma^{-1})$ for all $\gamma\in G$;
\item for every $x\in G^{(0)}$, every finite family
$\gamma_1,\dots,\gamma_n\in G^x$, and every family
$c_1,\dots,c_n\in \bC$ satisfying $\sum_{i=1}^n c_i=0$, one has
\[
\sum_{i,j=1}^n \overline{c_i}c_j\,\psi(\gamma_i^{-1}\gamma_j)\le 0.
\]
\end{enumerate}
\end{definition}

\begin{remark}
Despite the terminology, a function of negative type need not take negative
values. The word ``negative'' refers to the quadratic form in
\textup{(iii)}, not to the pointwise sign of $\psi$.
\end{remark}

\begin{definition}
A continuous function $\psi\colon G\to \bR$ is said to be
\emph{locally proper} if for every compact subset $K\subseteq G^{(0)}$,
the restriction of $\psi$ to
\[
G_K^K:=r^{-1}(K)\cap s^{-1}(K)
\]
is proper.
\end{definition}

\begin{definition}
We say that $G$ has the \emph{Haagerup property} (or is
\emph{a-T-menable}) if it admits a continuous locally proper function of
negative type.
\end{definition}

\begin{remark}
If $\psi$ is a continuous function of negative type on $G$, then by
Schoenberg's theorem the functions
\[
h_t(\gamma):=e^{-t\psi(\gamma)}, \qquad t>0,
\]
are continuous positive definite functions on $G$. If moreover $\psi$ is
locally proper, then $h_t\to 1$ uniformly on compact subsets as $t\downarrow 0$.
Thus the Haagerup property provides a natural source of positive definite
multipliers approximating the identity.
\end{remark}

\begin{remark}
The Rapid Decay Property and the Haagerup property play different roles in our
argument. The Rapid Decay Property is formulated with respect to a length
function $L\colon G\to [0,\infty)$, while the Haagerup property is formulated
in terms of a negative type function $\psi\colon G\to \bR$. In general, these
are different pieces of data and should not be confused. The role of $\psi$ is
to produce positive definite multipliers, whereas the role of $L$ is to control
the Sobolev norms appearing in the Rapid Decay estimate.
\end{remark}

\subsection{Multipliers and approximations via RD}

If $h$ is a bounded continuous positive definite function on a groupoid $G$,
then
\[
M_h(f)(\gamma)=h(\gamma)f(\gamma),
\qquad f\in C_c(E),
\]
extends uniquely to a completely positive map on $C_r^*(E)$, and
\[
\|M_h\|=\|h\|_\infty,
\]
by \cite[Lemma~4.2]{Takeishi2014}; see also
\cite[Proposition~3.6]{KwasniewskiLiSkalski2022}.

\begin{proposition}\label{hap}
Let $E$ be a Fell bundle over a groupoid $G$, and let $(h_i)$ be a uniformly
bounded net of continuous positive definite functions on $G$ converging to $1$
uniformly on compact subsets of $G$. Then
\[
\|M_{h_i}(f)-f\|_r \to 0
\qquad\text{for all }f\in C_r^*(E).
\]
\end{proposition}

\begin{proof}
Fix $\varepsilon>0$, and let $M:=\sup_i\|h_i\|_\infty<\infty$.
Given $f\in C_r^*(E)$, choose $f_N\in C_c(E)$ such that
\[
\|f-f_N\|_r<\frac{\varepsilon}{3M}.
\]
Then
\[
\|M_{h_i}(f)-f\|_r
\le
\|M_{h_i}(f-f_N)\|_r
+
\|M_{h_i}(f_N)-f_N\|_r
+
\|f_N-f\|_r.
\]
Since $\|M_{h_i}\|=\|h_i\|_\infty\le M$, we get
\[
\|M_{h_i}(f-f_N)\|_r
\le
M\|f-f_N\|_r
<
\frac{\varepsilon}{3}.
\]
Also, by \cite[Lemma~4.5]{KwasniewskiLiSkalski2022},
\[
\|M_{h_i}(f_N)-f_N\|_r\to 0.
\]
Hence, for $i$ sufficiently large,
\[
\|M_{h_i}(f_N)-f_N\|_r<\frac{\varepsilon}{3}.
\]
Therefore, for all sufficiently large $i$,
\[
\|M_{h_i}(f)-f\|_r<\varepsilon.
\]
\end{proof}

\begin{proposition}\label{localschwartz-r}
Assume that $E$ has the Rapid Decay Property with respect to $L$, witnessed by
constants $C>0$ and $q\geq 0$.
If $f\in H^{2,L}(E)$ satisfies $\supp(f)\subseteq U$, where $U\subseteq G$ is open,
then for every $\varepsilon>0$ there exists $g\in C_c(E|_U)$ such that
\[
\|f-g\|_r<\varepsilon.
\]
Equivalently,
\[
f\in \overline{C_c(E|_U)}^{\|\cdot\|_r}.
\]
\end{proposition}

\begin{proof}
Since $f\in H^{2,L}(E)$, one has
\[
f(1+L)^q\in \ell^2_s(E)\cap \ell^2_r(E),
\]
and clearly
\[
\supp\bigl(f(1+L)^q\bigr)=\supp(f)\subseteq U.
\]
By Corollary~\ref{cor:schwartz-locality-II}, there exists a net
$(g_i)_i\subseteq C_c(E|_U)$ such that
\[
\|g_i-f\|_{2,q,L}\to 0.
\]
Applying Rapid Decay, we get
\[
\|g_i-f\|_r\le C\|g_i-f\|_{2,q,L}\to 0.
\]
Hence, for every $\varepsilon>0$, some $g_i\in C_c(E|_U)$ satisfies
\[
\|f-g_i\|_r<\varepsilon.
\]
\end{proof}

\begin{lemma}\label{lemma:multiplier-into-schwartz}
Let $E$ be a Fell bundle over an \'etale groupoid $G$, and assume that $E$
has the Rapid Decay Property with respect to a length function $L$ on $G$.

Let $h\colon G\to \bC$ be a bounded continuous function such that, for every
$p>0$,
\[
B_p(h):=\sup_{\gamma\in G}|h(\gamma)|(1+L(\gamma))^p<\infty.
\]
Then, for every $f\in C_r^*(E)$, one has
\[
M_h(f)\in H^{2,L}(E).
\]
Moreover, for every $p>0$,
\[
\|M_h(f)\|_{2,p,L}\leq B_p(h)\|f\|_r.
\]
\end{lemma}

\begin{proof}
Fix $p>0$ and $f\in C_r^*(E)$.
Choose a sequence $(f_n)\subseteq C_c(E)$ such that
\[
\|f_n-f\|_r\to 0.
\]
Since $M_h$ is bounded on $C_r^*(E)$, we have
\[
\|M_h(f_n)-M_h(f)\|_r\to 0.
\]

If $m,n$ are arbitrary, then
\[
M_h(f_n-f_m)(\gamma)=h(\gamma)(f_n-f_m)(\gamma),
\]
so
\[
(M_h(f_n-f_m))(1+L)^p=h(1+L)^p(f_n-f_m).
\]
Using the definition of the $II$-norm and the estimate
\[
|h(\gamma)|(1+L(\gamma))^p\leq B_p(h),
\]
we obtain
\[
\|M_h(f_n-f_m)\|_{2,p,L}
\leq B_p(h)\|f_n-f_m\|_{II}.
\]
By \eqref{eq:norm-inequalities},
\[
\|f_n-f_m\|_{II}\leq \|f_n-f_m\|_r,
\]
hence
\[
\|M_h(f_n-f_m)\|_{2,p,L}
\leq B_p(h)\|f_n-f_m\|_r.
\]
Thus $(M_h(f_n))_n$ is Cauchy in \(\|\cdot\|_{2,p,L}\).

Let \(g_p\) denote its limit in \(H^{2,p,L}(E)\).
On the other hand,
\[
\|M_h(f_n)-M_h(f)\|_\infty\le \|M_h(f_n)-M_h(f)\|_r\to 0,
\]
so \(M_h(f_n)\to M_h(f)\) uniformly in \(C_0(E)\). Therefore \(g_p=M_h(f)\)
as an element of \(C_0(E)\). Since this holds for every \(p>0\), we conclude
that
\[
M_h(f)\in \bigcap_{p>0} H^{2,p,L}(E)\cap C_0(E)=H^{2,L}(E).
\]

Finally, passing to the limit in the above estimate gives
\[
\|M_h(f)\|_{2,p,L}\leq B_p(h)\|f\|_r.
\]
\end{proof}

\begin{theorem}\label{thm:localizability-via-rd}
Let $E$ be a Fell bundle over an \'etale groupoid $G$, and let
$U\subseteq G$ be open.
Assume that $E$ has the Rapid Decay Property with respect to a length function
$L$ on $G$.

Assume moreover that there exists a net $(h_i)_i$ consisting of continuous positive definite functions on $G$ such that:
\begin{enumerate}[label=\textup{(\roman*)}]
\item $\|h_i\|_\infty \leq c$ for all $i$, for some $c>0$; 
\item $h_i\to 1$ uniformly on compact subsets of $G$;
\item for every $p>0$ and every $i$,
\[
B_{p,i}:=\sup_{\gamma\in G}|h_i(\gamma)|(1+L(\gamma))^p<\infty.
\]
\end{enumerate}

If $f\in C_r^*(E)$ satisfies
\[
\supp(f)\subseteq U,
\]
then there exists a net $(f_i)\subseteq C_c(E|_U)$ such that
\[
\|f_i-f\|_r\to 0.
\]
Equivalently,
\[
\{a\in C_r^*(E):\supp(a)\subseteq U\}
=
\overline{C_c(E|_U)}^{\|\cdot\|_r}.
\]
\end{theorem}

\begin{proof}
Let \(f\in C_r^*(E)\) satisfy \(\supp(f)\subseteq U\).

By Proposition~\ref{hap},
\[
\|M_{h_i}(f)-f\|_r\to 0.
\]
For each \(i\), Lemma~\ref{lemma:multiplier-into-schwartz} gives
\[
M_{h_i}(f)\in H^{2,L}(E),
\]
and clearly
\[
\supp(M_{h_i}(f))\subseteq \supp(f)\subseteq U.
\]

Applying Proposition~\ref{localschwartz-r} to \(M_{h_i}(f)\), we obtain that for
every \(\varepsilon>0\) there exists \(g_{i,\varepsilon}\in C_c(E|_U)\) such that
\[
\|M_{h_i}(f)-g_{i,\varepsilon}\|_r<\varepsilon.
\]

Choose a net \((\varepsilon_i)\) with \(\varepsilon_i\to 0\), and for each \(i\)
pick \(f_i\in C_c(E|_U)\) such that
\[
\|M_{h_i}(f)-f_i\|_r<\varepsilon_i.
\]
Then
\[
\|f_i-f\|_r
\le
\|f_i-M_{h_i}(f)\|_r+\|M_{h_i}(f)-f\|_r\to 0.
\]
Therefore there exists a net \((f_i)\subseteq C_c(E|_U)\) such that
\[
\|f_i-f\|_r\to 0.
\]

This proves
\[
\{a\in C_r^*(E):\supp(a)\subseteq U\}
\subseteq
\overline{C_c(E|_U)}^{\|\cdot\|_r}.
\]
The reverse inclusion is immediate, since every element of \(C_c(E|_U)\) has
support contained in \(U\), and this property is preserved under reduced-norm
limits inside \(C_0(E)\).
\end{proof}

\begin{corollary}\label{cor:localizability-via-negative-type}
The conclusion of Theorem~\ref{thm:localizability-via-rd} holds whenever
$E$ has the Rapid Decay Property with respect to a length function $L$, and
there exists a continuous locally proper negative type function
$\psi\colon G\to \bR$ such that, for
\[
h_t(\gamma):=e^{-t\psi(\gamma)},\qquad t>0,
\]
one has
\[
\sup_{\gamma\in G}|h_t(\gamma)|(1+L(\gamma))^p<\infty
\]
for every \(p>0\) and every \(t>0\).
\end{corollary}

\begin{proof}
Since \(\psi\) is of negative type, Schoenberg's theorem implies that each
\[
h_t(\gamma)=e^{-t\psi(\gamma)}
\]
is continuous and positive definite on \(G\).
Since \(\psi\) is locally proper and vanishes on \(G^{(0)}\), one has
\(h_t\to 1\) uniformly on compact subsets of \(G\) as \(t\downarrow 0\).

By assumption,
\[
\sup_{\gamma\in G}|h_t(\gamma)|(1+L(\gamma))^p<\infty
\]
for every \(p>0\) and every \(t>0\).
Thus the family \((h_t)_{t>0}\) satisfies the hypotheses of
Theorem~\ref{thm:localizability-via-rd}, and the conclusion follows.
\end{proof}


\end{document}